\documentclass[12pt]{article} 

\usepackage{amsfonts}
\usepackage{amssymb}
\usepackage{amsmath,amsthm}
\usepackage{latexsym}
\usepackage{amsbsy}
\usepackage{amscd}
\usepackage{color}
\usepackage[all]{xy}
\CompileMatrices

\newtheorem{teorema}{Theorem}[section]
\newtheorem{proposicion}[teorema]{Proposition}
\newtheorem{lema}[teorema]{Lemma}
\newtheorem{corolario}[teorema]{Corollary}
\newtheorem{ejemplo}[teorema]{Example}
\newtheorem{ejemplos}[teorema]{Examples}
\newtheorem{definicion}[teorema]{Definition}
\newtheorem{nota}[teorema]{Remark}

\newtheorem{proyecto}{Project}

\def\bp{B(\P^3,2)}
\def\homeo{\P^3\times\left(\P^3-\{e\}\right)}
\def\fade{\P^3-\{e\}}
\def\cat{\mathrm{cat}}
\def\TC{\mathrm{TC}}
\def\F2{\mathbb{F}_2}
\def\Z2{\Sigma_2}
\def\Sq{\mathrm{Sq}}
\def\gen{\mathfrak{genus}}

\def\map{\mathrm{Map}}
\def\ev{\mathrm{ev}}
\def\emb{\mathrm{Emb}}
\def\e{{\cal E}}
\def\imm{\mathrm{Imm}}
\def\immto{\looparrowright}
\def\P{\mathrm{P}}
\title{Symmetric topological complexity
as the first obstruction in Goodwillie's 
Euclidean embedding tower
for real projective spaces}
\author{Jes\'us Gonz\'alez\footnote{Partially supported  
by CONACYT Research Grant 102783.}}
\date{\empty}

\begin{document}

\maketitle

\begin{abstract}
This paper explains why Goodwillie-Weiss 
calculus of embeddings can offer new information
about the Euclidean embedding dimension of $\P^m$ only for $m\leq15$.
Concrete scenarios are described in these 
low-dimensional cases, pinpointing where to look for potential---but 
critical---high-order obstructions in the corresponding Taylor towers.
For $m\geq16$, the relation $\TC^S(\P^m)\geq n$ is translated 
into the triviality of a certain cohomotopy Euler 
class which, in turn, becomes the only 
Taylor obstruction to producing an embedding $\P^m\subset\mathbb{R}^n$. A
speculative bordism-type form of this primary obstruction is proposed
as an analogue of Davis' $BP$-approach to the immersion problem of 
$\P^m$. A form of the Euler class viewpoint  is applied to show 
$\TC^S(\P^3)=5$, as well as to suggest a few higher dimensional 
projective spaces for which the method could produce new information.
As a second goal, the paper extends
Farber's work on the motion planning problem in order to develop 
the notion of a symmetric motion planner for a mechanical system ${\cal S}$.
Following Farber's lead, this concept is connected to $\TC^S(C({\cal S}))$,
the symmetric topological complexity of the state space 
of ${\cal S}$. The paper ends by sketching the construction
of a concrete $5$-local-rules symmetric motion planner for $\P^3$.
\end{abstract}



\section{Main results}
\label{sectctcsgoodwillie}
Recall the Schwarz genus $\gen(p)$ 
of a fibration $p\colon E\to B$~(\cite{schwarz});
it is one less\footnote{Normalization is chosen so that a trivial
fibration has $\gen=0$.} than
the smallest number of open
sets $U$ covering $B$ in such a way that $p$ 
admits a (continuous) section over each $U$. 

\begin{definicion}[\cite{F1}]\label{TCdef}
The topological complexity of 
a space $X$, $\TC(X)$, is defined as the genus of the end-points evaluation
map $\ev\colon P(X)\to X\times X$, where $P(X)$ is the free path
space $X^{[0,1]}$ with the compact-open topology. 
\end{definicion}

Let $F(X,2)\subset X\times X$ denote the configuration space of ordered 
pairs of distinct points in $X$, 
and $\ev_1\colon P_1(X)\to F(X,2)$ be the restriction of 
the fibration $\ev$. Thus $P_1(X)$ is the subspace of $P(X)$ 
obtained by removing the free loops on $X$. The group $\mathbb{Z}/2$
acts freely on both $P_1(X)$ and $F(X,2)$, by 
running a path backwards in the former, and by 
switching coordinates in the latter. Furthermore, $\ev_1$ is a
$\mathbb{Z}/2$-equivariant map. Let $P_2(X)$
and $B(X,2)$ denote the corresponding orbit spaces,
and let $\ev_2\colon P_2(X)\rightarrow B(X,2)$ denote
the fibration induced by $\ev_1$. 

\begin{definicion}[\cite{FGsymm}]\label{TCSdef}
The symmetric topological complexity of $X$, $\TC^S(X)$, is defined by
$\TC^S(X)=\gen(\ev_2)+1$.
\end{definicion}

These constructions arise in connection with the motion 
planning problem in robotics. The reader interested in the
motivation and further calculations should consult~\cite{F1,FGsymm} and
references therein.

\smallskip
The concept of a symmetric motion planner (as a collection of
suitably chosen continuous $\mathbb{Z}/2$-equivariant local 
sections---the local rules---for $\ev_1$) for a space, e.g.~the
state space of an autonomous robot, is formalized in Section~\ref{smp} 
in order to prove:

\begin{teorema}\label{tcsvsslr}
If $X$ is a smooth manifold, then $\TC^S(X)$ is the smallest 
possible number of local rules 
of symmetric motion planners for $X$.
\end{teorema}

The following result is derived in Section~\ref{secexa} 
by applying the ideas in Section~\ref{secprojuno}.

\begin{teorema}\label{vendimia}
$\TC^S(\mathrm{SO}(3))=5$. Consequently, $5$ is the smallest possible 
number of local rules in any symmetric motion planner of an 
autonomous robot whose space of states is described by all 
three-dimensional rotations.
\end{teorema}

Section~\ref{smp} describes an explicit $5$-local-rules motion planner
for an autonomous robot whose space of states is $\mathrm{SO}(3)$.
The main ingredient in the proof of Theorem~\ref{vendimia}
comes from the following calculation, the main goal in Section~\ref{secexa}:

\begin{teorema}\label{lacohofinal} The integral cohomology of 
$B(\mathrm{SO}(3),2)$ is 
$$H^*(B(\mathrm{SO}(3),2))
=\begin{cases}\mathbb{Z},&*=0;\\0,&*=1\mbox{ \ or \ }*\geq6;\\
\mathbb{Z}/2\oplus\mathbb{Z}/2,&*=2;\\\mathbb{Z}\oplus\mathbb{Z}/2,&*=3;\\
\mathbb{Z}/4,&*=4;\\\mathbb{Z}/2,&*=5.
\end{cases}$$
Furthermore, the classifying map of the double cover
$F(\mathrm{SO}(3),2)\to B(\mathrm{SO}(3),2)$ pulls back the square of the 
generator in $H^2(\P^\infty)$ to the element of order $2$ in 
$H^4(B(\mathrm{SO}(3),2))$.
\end{teorema}

\begin{nota}\label{Fred}{\em
Note that $H^*(B(\mathrm{SO}(3),2))$ has both $2$-torsion and $4$-torsion, 
but no $8$-torsion. Fred Cohen has indicated to the author
(even before the latter completed the proof of Theorem~\ref{lacohofinal})
that this seems to be the situation for all $B(\P^m,2)$ with $m\geq 2$.
Here and in what follows $\P^m$ stands for the $m$-dimensional
real projective space.
}\end{nota}

Some of the properties in the 
multiplicative structure of the cohomology ring of 
$B(\mathrm{SO}(3),2)$ are discussed in Remark~\ref{estrult}.

\medskip
Theorem~\ref{vendimia} should be compared
to the known equality $\TC(\mathrm{SO}(3))=3$.
Actually, $\TC(G)=\mathrm{cat}(G)$ for any connected topological group 
$G$~(\cite{farberinstabilities}). On the other hand, it is known
that $\TC^S(\mathrm{SO}(2))=2$ (see Table~\ref{tabla} in 
Section~\ref{redsymmprod}), while $\TC(\mathrm{SO}(2))=1$.
An interesting open question arises, then, in regard to the deviation of
$\TC^S$ from $\TC$ for a general connected topological (or even Lie) 
group---it should be noted that the inequality $\TC^S\geq\TC$ is 
known from~\cite{FGsymm}.

\smallskip
Remark~\ref{suerte} in Section~\ref{redsymmcalc} provides some 
indirect evidence toward the possibility that the 
method of proof of Theorem~\ref{vendimia} could actually be used 
in order to compute $\TC^S(\P^m)$ for (at least) $m=5,6$. 
The calculations in Section~\ref{secexa} (see in particular Remark~\ref{razon})
seem to suggest that the potential phenomenon mentioned in Remark~\ref{Fred}
would play a critical role in the expected derivation of
new information about $\TC^S(\P^m)$. In fact, the author hopes that this
paper motivates a renewed interest in understanding the homotopy
properties of the unordered configuration spaces $B(\P^m,2)$, 
which are key objects in the Euclidean embedding problem for real 
projective spaces.

\section{$\TC$, $\TC^S$, and calculus of functors of embeddings}
\label{calculus} Despite their robotics origin,
the topological complexity ideas are closely related, in the case of $\P^m$, 
to the calculus of embeddings developed in~\cite{GW,W}
by Goodwillie and Weiss. This is the main 
motivation for the present paper. The explicit connection is explained 
right after taking a quick glance at Goodwillie-Weiss' framework.

\smallskip
For a smooth $m$-dimensional manifold $M$ let 
$\e_n^M=\emb(-,\mathbb{R}^n)$ denote the functor 
of smooth embeddings of open sets of $M$ into $\mathbb{R}^n$. 
The Taylor expansion of $\e_n^M$ is a tower 
of functors 
\begin{equation}\label{tower}
\cdots\to{\cal T}_k\e_n^M
\stackrel{r_k}{\to}\cdots\stackrel{r_3}{\to}{\cal T}_2\e_n^M\stackrel{r_2}{\to}
{\cal T}_1\e_n^M
\end{equation}
equipped with 
compatible mappings $\eta_k\colon \e_n^M\to{\cal T}_k\e_n^M$
which are to be thought as giving better and better (as $k$ increases)
approximations of $\e_n^M$. For instance, the standard model for 
the linear approximation ${\cal T}_1\e_n^M$ is the sheaf 
of smooth immersions into $\mathbb{R}^n$. In this case
$\eta_1$ is the obvious inclusion, an equivalence through
dimensions $n-2m+1$ (Whitney's stable range). Likewise, for
a manifold $M$ with 
\begin{equation}\label{metarange}
2n\geq 3(m+1)
\end{equation}
(recall $\dim(M)=m$), Haefliger's work shows that a model 
for the quadratic approximation ${\cal T}_2\e_n^M(M)$ is given by
the space $\map^{\mathbb{Z}/2}(F(M,2), S^{n-1})$ of (continuous) 
$\mathbb{Z}/2$-equivariant maps $F(M,2)\to S^{n-1}$
(with antipodal $\mathbb{Z}/2$-action on the sphere $S^{n-1}$). 
This time $\eta_2(f)$ sends a pair $(x,y)\in F(M,2)$ into the normalized
difference
\begin{equation}\label{haefligermap}
\frac{f(x)-f(y)}{||f(x)-f(y)||}\in S^{n-1},
\end{equation}
determining  
an equivalence $\eta_2$ through dimensions $2n-3m-3$.
More generally, in a certain range the tower~(\ref{tower}) is an analytic
approximation for $\e_n^M$ since, according to~(\cite{GKW,GW}), for 
$n>m+2$
\begin{equation}\label{equivalencerange}
\mbox{$\eta_k$ is an equivalence through dimensions }\,kn-(k+1)m-2k+1.
\end{equation} 

Throughout the rest of this section
$M$ is replaced by the $m$-dimensional real projective space $\P^m$.
The first point to explain is how the work in~\cite{FTY,symmotion} 
(which are the basic references for the following assertions)
allows us to think of $\TC$ and $\TC^S$ as giving
detecting indicators for the linear and quadratic
terms in the Taylor tower for $\e_n^{\P^m}$. To begin with, if
${\cal T}_1\e_n^{\P^m}(\P^m)$ is non-empty, then
\begin{equation}\label{FTY}\TC(\P^m)\leq n.\end{equation}
Likewise, if ${\cal T}_2\e_n^{\P^m}(\P^m)$ is 
non-empty, then
\begin{equation}\label{GL}\TC^S(\P^m)\leq n.\end{equation}
Furthermore, 
not only do both implications hold without range restrictions but, 
except for a very limited number of cases, they are sharp.
For instance, the implication involving~(\ref{FTY}) is reversible
for a non-parallelizable $\P^m$, and in
the three exceptional cases of parallelizability ($m=1,3,7$) one has
$${\cal T}_1\e_n^{\P^m}(\P^m)\neq\varnothing\;\;\Leftrightarrow\;\;
\TC(\P^m)\leq n-1.$$

The discussion of the
situation for the implication involving~(\ref{GL})---the main topic
in~\cite{symmotion}---is perhaps more interesting. 
Namely, such an implication is reversible, for instance,
when $\TC^S(\P^m)$ is within Haefliger's metastable range~(\ref{metarange}), 
that is, when 
\begin{equation}\label{tcsenelmeta}
2\,\TC^S(\P^m)\geq 3(m+1).
\end{equation}
In fact, Haefliger's work gives\footnote{Brian Munson's help, 
starting at the 2009 CBMS conference in Cleveland, 
has been fundamental for realizing the correct form of this fact.} 
$$
\emb(\P^m,\mathbb{R}^n)\neq\varnothing\;\;\Leftrightarrow\;\;
{\cal T}_2\e_n^{\P^m}(\P^m)\neq\varnothing
\;\;\Leftrightarrow\;\;\TC^S(\P^m)\leq n
$$
whenever~(\ref{tcsenelmeta}) holds, so that
$\TC^S(\P^m)$ agrees with $\emb(\P^m)$, the 
embedding dimension of $\P^m$. 
In other words, within the range determined by~(\ref{tcsenelmeta}), 
if a given immersion 
$f\colon\P^m\immto\mathbb{R}^n$ can be regularly deformed into one coming from
an element in ${\cal T}_2\e_n^{\P^m}(\P^m)$, then the immersion can be 
regularly deformed into an embedding. An alternative
phrasing (which ignores indeterminacy issues) of this fact goes as follows: 
When~(\ref{tcsenelmeta}) holds, there is only one 
obstruction to lifting elements in 
${\cal T}_1\e_n^{\P^m}(\P^m)$ through the several maps $r_k$ in~(\ref{tower});
that obstruction holds at $k=2$, has to do with avoiding double points,
and is numerically detected by the condition $\TC^S(\P^m)\leq n$.

\smallskip
It is known from~\cite{symmotion} that Haefliger's metastable range 
condition~(\ref{tcsenelmeta}) holds for $m\ge16$, 
as well as for $m=4,8,9,10,13$
(for the case $m=10$, see Table~\ref{tabla} in Section~\ref{redsymmprod}, 
and the comments therein). 
But the equality $\TC^S(\P^m)=\emb(\P^m)$
is also known for $m\le2$
and, from Theorem~\ref{vendimia}, for $m=3$.
However, it is currently unknown if in any of the remaining cases,
\begin{equation}\label{exceptions}m\in\{5,6,7,11,12,14,15\},\end{equation}
\begin{itemize}
\item[(a)] the implication involving~(\ref{GL}) is reversible,
\item[(b)] there are higher obstructions to lifting through the
corresponding tower~(\ref{tower}), or if, on the contrary, 
\item[(c)] $\TC^S(\P^m)$ agrees with the embedding dimension of $\P^m$. 
\end{itemize}

The major goal in this paper 
is to suggest ways to shed light on these low dimensional cases---mainly 
through Projects~\ref{problema1} and~\ref{problema3} below, although 
Project~\ref{problema2} has interesting potential implications for higher
dimensional projective spaces (see item~II in Section~\ref{lasecciontres}).

\smallskip
The section closes by recording the result of peeling Haefliger's 
numerical restriction~(\ref{tcsenelmeta}) off the relation between
${\cal T}_2\e_n^{\P^m}(\P^m)$ and $\TC^S(\P^m)$. Explicit computations
in the case $m=3$ can be found
at the end of Section~\ref{smp}.

\begin{teorema}[\cite{symmotion}]\label{mainsymmotion}
$\map^{\mathbb{Z}/2}(F(\P^m,2), S^{n-1})\neq\varnothing$ if and only if 
$\:\TC^S(\P^m)\leq n\,$.
\end{teorema}

\begin{nota}\label{obstr1}{\em
Haefliger's map~(\ref{haefligermap}), $\eta_2\colon\emb(M,\mathbb{R}^n)\to 
\map^{\mathbb{Z}/2}(F(M,2), S^{n-1})$, implies that the inequality 
$\TC^S(\P^m)\leq n$ is a consequence of
the existence of an embedding $\P^m\subset\mathbb{R}^n$.
}\end{nota}

The rest of the paper is organized as follows. 
Sections~\ref{secprojuno}--\ref{secprojcinco}
describe four projects aiming to get 
a hold on $\TC^S(\P^m)$ and $\emb(\P^m)$. 
Potential examples on the use of these projects are 
given in Section~\ref{redsymmprod}.
The reader interested in the proof of
Theorems~\ref{vendimia} and~\ref{lacohofinal} 
should proceed directly to 
Section~\ref{secexa}, going back as needed (explicit 
cross-references are given) to Section~\ref{secprojuno}
for the general strategy, and to Section~\ref{redsymmcalc}
for a preliminary $\F2$-cohomology calculation 
due to Haefliger. 
The final Section~\ref{smp} is devoted to (a) developing the concept of 
a symmetric motion planner for a mechanical system, (b) studying its 
relation with the symmetric topological complexity of the space of states
of the system (proof of Theorem~\ref{tcsvsslr}), 
and (c) sketching the construction of a concrete symmetric motion 
planner, with the least possible number of local rules, when the 
state space of the system is $\mathrm{SO}(3)$.

\bigskip\noindent 
{\bf Acknowledgments}

\smallskip
The author wishes to thank Fred Cohen and Peter Landweber
for many useful suggestions for computing the integral 
cohomology ring of the unordered configuration space
of two distinct points in $\P^m\,$---a central object for next section's
purposes. Cohen's suggestion of using the Bockstein spectral 
sequence approach eventually led the author to the calculations described 
at the end of Examples~\ref{ejemp2},
clarifying a subtle point in the main Cartan-Leray-Serre spectral sequence
calculation of Section~\ref{secexa}.
The author wishes to specially thank Peter Landweber for carefully
reading this manuscript and for many fine suggestions throughout the
development of this paper. In particular,
Section~\ref{smp} arose from an observation of Landweber on an
earlier version of Theorem~\ref{vendimia}.

\section{The first project}\label{secprojuno}
Theorem~\ref{mainsymmotion} can be restated by saying that
$\TC^S(\P^m)$ is the smallest positive integer $n$ such that
the classifying map for the line
bundle $\zeta_m$ associated to the canonical 
$\mathbb{Z}/2$-principal projection
$F(\P^m,2)\to B(\P^m,2)$ admits a homotopy compression
\begin{equation}\label{compression}
\zeta\colon B(\P^m,2)\to\P^{n-1}.
\end{equation}
A Hopf-type approach to this 
problem is then given by 
analyzing, with mod $2$ singular cohomology, the algebraic
possibilities for a potential map~(\ref{compression}). 
Namely, the vanishing of the $n$-th power of 
$w_1(\zeta_m)$ is a necessary condition for the existence of a 
compressed map $\zeta$ as 
in~(\ref{compression})---and thus for an embedding 
$\P^m\subset\mathbb{R}^n$. The method was known in
the 60's~(\cite{handel}) and, during the subsequent decade,
it was systematically used 
to evaluate, within Haefliger's metastable range,
groups of embeddings of low efficiency~(\cite{bausum,feder,yasui1,yasui2})
for a given manifold $M$. The main input for those computations 
(and their answer) is given by the integral cohomology groups of $B(M,2)$
with both simple and local coefficients.

\medskip
The interest in this section focuses on $M=\P^m$, 
keeping more of the original motivation 
in~\cite{handel}, but without having to restrict attention
to Haefliger's metastable range~(\ref{tcsenelmeta}).
For instance, the explicit cases worked out 
at the end of Section~\ref{redsymmcalc}
allow us to recover, for $m\leq8$,  the lower bounds
for $\TC^S(\P^m)$ obtained in~\cite{symmotion} using ideas in~\cite{FGsymm}.
But integral coefficient cohomology {\it rings} 
are now central, and even more interesting is the possibility that,
by replacing singular cohomology by suitable generalized
cohomology theories, one could get
important improvements on the lower bounds for $\TC^S(\P^m)\,$---which, 
in view of Remark~\ref{obstr1}, might even lead to new nonembedding 
results for $\P^m$. A number of potential situations 
of this sort are described in Section~\ref{redsymmprod}.
The idea is in the same spirit as the one exploited 
in~\cite{Astey,Davistrong}
(for immersions rather than embeddings) using forms of complex cobordism
instead of singular cohomology---compare to
item~II in Section~\ref{lasecciontres}.

\smallskip
In this direction, the first proposed project is:

\begin{proyecto}\label{problema1}{\em\ 

\vspace{-2mm}
\begin{enumerate}
\item\label{p1.1} Describe (enough of) the integral cohomology 
{\em ring} $H^*(B(\P^m,2))$.
\item\label{p1.2} (Profit:)~Apply the resulting information 
within the Hopf-type approach described right after~(\ref{compression}) 
to get (new, perhaps optimal in
certain low dimensional cases) lower bounds for 
$\TC^S(\P^m)\,$---and potentially for $\emb(\P^m)$.
\end{enumerate}
}\end{proyecto}

The case of mod $p$ coefficients in part~\ref{p1.1} of Project~\ref{problema1}
was treated by Haefliger~\cite{haefliger} ($p$=2) and F\'elix-Tanr\'e~\cite{FT}
(odd prime $p$). Haefliger's method is reviewed in Section~\ref{redsymmcalc}.
The information is used in Section~\ref{secexa}, following 
the ideas in Project~\ref{problema1}, to compute $\TC^S(\P^3)$.

\smallskip
As far as the author knows, the integral cohomology ring---or, for 
that matter, its additive structure---required in 
part~\ref{p1.1} of Project~\ref{problema1} is currently 
unknown (\cite{bausum} gives a theoretical description of the 
Bockstein spectral sequence for $B(\P^m,2)$ when $m$ is odd, 
but the explicit calculations get combinatorially out of hands 
as $m$ increases). Information about
this ring would in addition pave the way for the corresponding calculation
in terms of some (advantageous for the Hopf-type approach) 
generalized cohomology theory---e.g.~complex 
cobordism. Furthermore, information on the integral cohomology 
ring $H^*(B(\P^m,2))$
is not only usable as described in Project~\ref{problema1}, but it
would lay the ground for Project~\ref{problema3}, which is presented next.

\section{The second project}
Since $\P^{n-1}$ classifies line bundles whose 
$n$-th Whitney multiple admits a nowhere zero section~(\cite{sam}),
the considerations in~(\ref{compression}) imply that 
Theorem~\ref{mainsymmotion} can be recast in the following terms:  
{\em $\TC^S(\P^m)$ is the smallest positive integer $n$ such that $n\zeta_m$
admits a nowhere zero section.} In particular,
the Hopf-type approach in the previous project can be interpreted
in terms of the observation that (generalized) Euler classes give initial
obstructions to sectioning $n\zeta_m$ (indeed, this is the 
starting point in the classical obstruction theory 
approach of~\cite{handel}). For a given cohomology 
theory $E$, much of the problem in the Euler class 
setting comes from the fact that, due to a possible lack of $E$-orientability
of $n\zeta_m$ (or $\zeta_m$, for that matter),
the calculation of the $E$-Euler class $\chi_E(n\zeta_m)$
might become a very difficult task---but note that generalized 
Euler classes are defined even without orientability 
assumptions~(\cite{crabb}). This difficulty is largely redeemed by the nice
fact that the cohomotopy Euler class is the unique 
obstruction for the mono-sectioning problem of bundles of relatively 
large dimension. Concretely, as indicated in~\cite{crabb}, a
bundle $\alpha$ over a complex $X$ satisfying 
\begin{equation}\label{crabbstable}
2\dim(\alpha)\ge\dim(X)+3
\end{equation}
admits a nowhere zero section if and only if the cohomotopy Euler
class $\chi_{S^0}(\alpha)$ vanishes. This fact is very 
effectively put to work by Stolz
in~\cite{stolz} in order to give a complete description of
the first Whitney multiple of $\lambda_m$ admitting a nowhere zero section,
where $\lambda_m$ is the canonical line bundle over a given $\P^{2m+1}$.

\medskip
Since an embedding $\P^m\subset\mathbb{R}^n$ can exist only with $n>m$,
and since
\begin{equation}\label{finito}
\mbox{$B(\P^m,2)$ has the homotopy type of a closed 
$(2m-1)$-dimensional manifold}%complex of dimension at most $2m-1$}
\end{equation}
(see~\cite[Proposition~2.6]{handel}),
condition~(\ref{crabbstable}) holds in all relevant instances
of $\alpha=n\zeta_m$. This situation can now be used to complement 
Project~\ref{problema1}, whose long term goal can be thought of as
finding suitable cohomology theories $E$ where $\chi_E(n\zeta_m)\neq0$ 
(thus obtaining lower bounds for $\TC^S(\P^m)$ and, consequently, 
for $\emb(\P^m)$). Indeed, the next project is concerned with
getting upper bounds for $\TC^S(\P^m)$ by means of 
the exploration of the cohomotopy Euler class of $n\zeta_m$ for suitable
values of $n$. 

\begin{proyecto}\label{problema3}{\em \ 

\vspace{-2mm}
\begin{enumerate}
\item Adapt Stolz's method in order to identify instances 
where the cohomotopy Euler class of $n\zeta_m$ vanishes.
\item (Profit:)~This would produce upper bounds for the symmetric 
topological complexity of projective spaces, as well as  
(potentially new) embeddings of those manifolds 
in Haefliger's metastable range.
\end{enumerate}
}\end{proyecto}

Stolz's method for studying cohomotopy Euler classes 
makes essential use of the calculability of
the corresponding integral (singular) cohomology Euler classes.
It is in this sense that the integral cohomology calculations needed
for Project~\ref{problema1} would also find application in 
Project~\ref{problema3}. 

\section{The third project}\label{lasecciontres}
While the numerical behavior of $\TC^S(\P^m)$ is the central issue 
in the first two projects,
the next two will concern, instead, the higher obstructions 
one could meet in the Taylor tower~(\ref{tower}) for $M=\P^m$.
Thus, the next project can be thought of an extension
of (the cobordism approach in) Project~\ref{problema1}.

\medskip
Start by recalling Hirsch's interpretation of 
the immersion problem for a closed
manifold $M$ in terms of the geometric 
dimension of its stable normal bundle $\nu_M$: the 
minimal codimension of Euclidean immersions of $M$ 
agrees with the smallest positive integer $d$ for which the map 
$M\to BO$ classifying $\nu_{M}$  
admits a homotopy factorization $$M\to BO(d)\hookrightarrow BO.$$
Obstruction theory then decomposes this task into small goals 
in which one has to deal with the $k$-invariants
in a suitably flavored Postnikov tower for (a fibration homotopy 
equivalent to) the inclusion $BO(d)\hookrightarrow BO$.
Such obstructions are theoretically easy to describe 
(they lie in singular cohomology 
groups), but are very hard to handle in practice (not only do they 
require information about the homotopy 
fiber of $BO(d)\hookrightarrow BO$, but obstructions can appear in the tower
at a high level involving difficult indeterminacy issues). 
For real projective spaces, 
the latter problem was avoided in~\cite{Davistrong} (with remarkable results)
by considering just one obstruction, a key Euler class
in Brown-Peterson theory. This primary obstruction
captures a great deal of information, and can be handled efficiently by means 
of algebraic methods. Although there is a corresponding 
version of such an obstruction for the embedding (as opposed to immersion) 
dimension of projective spaces~(\cite{asteyobstr}), the
comments below (particularly item II) should be thought of as suggesting
the possibility of a primary obstruction for the embedding problem
of projective spaces, similar in spirit to that in~\cite{Davistrong}, 
but now giving a complete obstruction within Haefliger's metastable range.

\medskip
The question of identifying secondary 
(and higher order) cobordism obstructions in the
Euclidean embedding problem of a given manifold arises 
very naturally from the Taylor expansion point of view 
(in the context of regularly deforming a 
given immersion into an embedding). For instance, just as Haefliger
pointed out the primary obstruction to lifting an element
through $r_2$ in~(\ref{tower}), the (secondary) obstruction to lifting 
through $r_3$ lies in a certain twisted cobordism group~(\cite{munson}). 
Now, as explained in the comments prior to Theorem~\ref{mainsymmotion},
for projective spaces
such higher obstructions can only appear in the cases (outside 
Haefliger's metastable range) indicated in~(\ref{exceptions}). 
These assertions have two interesting (potential) consequences:
\begin{itemize}
\item[I.]
High-order analysis of
obstructions (secondary obstructions in particular) 
in the Taylor tower~(\ref{tower}) has good possibilities 
to detect new phenomena in the embedding problem of {\it low dimensional}
(read: {\it outside Haefliger's meta\-stable range})
projective spaces---see the concrete situations worked out in 
Section~\ref{redsymmprod}.
\item[II.] (Again with $M=\P^m$) If
it is possible to identify the Brown-Peterson Euler class
in the Hopf-type approach in Project~\ref{problema1} with Haefliger's 
%\marginpar{\tiny (Sam) Check
%\\Koschorke's\\MR0979341\\(90k:57039):\\single obstr\\in bordism\\for
%sectioning\\and embed's\\as in here.}
obstruction for lifting 
through $r_2$ in~(\ref{tower}), then this would
mean that, within Haefliger's metastable range (e.g.~for $m\ge16$),
this Brown-Peterson (primary) obstruction would be the {\em only one} for the 
embedding problem of projective spaces. It is very appealing to
compare this possibility with the ``feeling''
that the negative immersion results in~\cite{Davistrong} seem
to be closer to optimal than the general positive immersion
results currently known for projective spaces. But the advantage in the 
new proposed embedding setting is that the resulting 
obstruction would be complete, thus
providing formal support for the embedding analogue of the feeling 
mentioned above.
\end{itemize}

The most difficult part in the previous setting would seem to be the 
identification (and manipulation) of the cobordism groups
containing high-order obstructions (not to mention
the actual calculation of the obstructions). Thus, the new  
project is:

\begin{proyecto}\label{problema2}{\em \ 

\vspace{-2mm}
\begin{enumerate}
\item\label{p2.1} (In the direction of I above:)~Give an explicit 
($=\,$manageable) description 
(recall $M=\P^m$) of the 
cobordism group containing the obstruction for lifting, through
$r_3$, a given element in ${\cal T}_2{\cal E}_n^{\P^m}(\P^m)$. Since this 
question is intended for values of $m$ outside Haefliger's metastable range,
the required model for ${\cal T}_2{\cal E}_n^{\P^m}(\P^m)$ would not be the space
$\map^{\mathbb{Z}/2}(F(\P^m,2),S^{n-1})$, but the usual space
Iso$^{\mathbb{Z}/2}(\P^m\times\P^m,\mathbb{R}^n\times\mathbb{R}^n)$ of strict
isovariant maps.
\vspace{-2mm}
\item\label{p2.2} Identify instances where this cobordism group
vanishes.
\vspace{-2mm}
\item\label{p2.3} (Profit:)~Except for a few very low-dimensional projective
spaces (specified in Section~\ref{redsymmprod}), 
the resulting lifted map would be within the 3/4 range ($k=3$ 
in~(\ref{equivalencerange})), so the corresponding embedding would be for free.
\vspace{-2mm}
\item\label{p2.4} (Profit:)~On 
the other hand, if one could prove nontriviality of 
Munson's obstruction~\cite{munson} 
for {\it any} strict isovariant map, then this would 
certainly imply a corresponding nonembedding result.
\vspace{-2mm}
\item\label{p2.5} (In the direction of II above:)~Sort out the hoped-for 
identification of primary 
obstructions mentioned in II above. If this does not work,
then find an {\em algebraic} characterization 
(some sort of cobordism group?)~of Haefliger's double obstruction
for lifting through $r_2$---(profit:)~such an
obstruction is then the only one,
within Haefliger's metastable range, for the embedding problem of 
projective spaces.
\end{enumerate}
}\end{proyecto}

For part~\ref{p2.2} of Project~\ref{problema2} to make better sense, one
would need to be specific about how to identify 
good ``instances'', i.e., those with high chances for Munson's secondary 
obstruction to vanish. 
Indeed, one needs to know {\it where} to
look for potential isovariant maps (i.e.~those in part~\ref{p2.1} of 
Project~\ref{problema2}) with trivial Munson's secondary obstruction. 
In view of Theorem~\ref{mainsymmotion} 
and Remark~\ref{obstr1}, a first approximation comes from 
the knowledge of the value of $\TC^S(\P^m)$---the motivation
for the two previous projects. And as a first step in this 
direction, Section~\ref{redsymmprod} 
describes what the author knows about concrete values of $\TC^S(\P^m)$ for low 
values of $m$, as well as about the expectations for high-order 
obstructions in~(\ref{tower}), and possible ways to manage them
by comparing with known information on the embedding dimension
of low-dimensional projective spaces (indeed, functoriality issues
might be helpful toward part~\ref{p2.4} of Project~\ref{problema2}).

\section{The fourth project}\label{secprojcinco}
The last project is of a more theoretical nature:
it has to do with finding models suitably ``approximating'' the
terms ${\cal T}_k{\cal E}_n^{M}$ in~(\ref{tower}). The idea is still in
a very crude stage, and is stated likewise.

\begin{proyecto}\label{problema4}{\em \ 
For $k\ge3$, find a space that models ${\cal T}_k{\cal E}_n^{M}(M)$ 
in a way that
resembles how $\map^{\mathbb{Z}/2}(F(M,2),S^{n-1})$
models ${\cal T}_2{\cal E}_n^{M}(M)$. 
}\end{proyecto}

Haefliger-type models for ${\cal T}_k{\cal E}_n^{M}$ have 
certainly been described in~\cite{GKWmodels}, but the author 
does not know how manageable those models are for concrete computations
(e.g.~with $M=\P^m$).
The idea here is to find alternative models which
can be handled more naturally from an algebraic topology point of 
view---even if this means concentrating on $M=\P^m$.

\section{Putting the projects to work}\label{redsymmprod}
Recall from Remark~\ref{obstr1} that $\emb(\P^m)$, the embedding dimension 
of $\P^m$, is an upper bound for $\TC^S(\P^m)$, and that 
equality actually holds when $\TC^S(\P^m)$ satisfies Haefliger's
metastable range condition~(\ref{tcsenelmeta})---e.g.~when $m\geq16$.
This section starts by describing what the author knows about $\TC^S(\P^m)$
and $\emb(\P^m)$ for low values of 
$m$ (potentially outside Haefliger's metastable range). Data
is summarized in Table~\ref{tabla},
where cases satisfying~(\ref{tcsenelmeta})
have been marked with a star.
\begin{table}
\centerline{
\begin{tabular}{|c|c|c|c|c|c|c|c|c|c|c|c|c|c|c|c|}\hline
$m$ & $1$ & $2$ & $3$ & $4^\star$ 
& $5$ & $6$ & $7$ & $8^\star$ & $9^\star$ 
& $10^\star$ & $11$ 
& $12$ & $13^\star$ &$14$ & $15$ \\ \hline
$\rule{0mm}{4mm}\emb(\P^m)$ & $2$ & $4$ & $5$ & $8$ & $9$ & $9..11$ & 
$9..12$ & $16$ & $17$ & 
$17$ & $17..18$ & $18..21$ & $22..23$ & $22..23$ & $23..24$ \\ \hline
$\rule{0mm}{4.3mm}\TC^S(\P^m)$ & $2$ & $4$ & ${\bf 5}$ & $8$ & $8..9$ 
& $8..9$ & $8..10$ & $16$ & 
$17$ & ${\bf 17}$ & ${\bf 17}..18$ & $18..21$ & $22..23$ 
& $22..23$ & $22..23$ \\ \hline
\end{tabular}}
\caption{$\emb(\P^m)$ vs.~$\TC^S(\P^m)$ for low values of $m$\label{tabla}}
\end{table}
The information is taken
from~\cite[Table~1 on page~480]{symmotion}, except for the boldface 
number 5  for $m=3$, and the two boldface numbers 17 
(which, after the required shift in notation, 
appear as 16 in~\cite{symmotion}) for $m=10,11$. The improvement for
$m=3$ is given by Theorem~\ref{vendimia}, while those for  $m=10,11$
follow from Theorem~\ref{mainsymmotion},
the known case $\TC^S(\P^9)=17$, and the $\mathbb{Z}/2$-equivariant 
inclusions $F(\P^9,2)\hookrightarrow F(\P^{10},2)\hookrightarrow F(\P^{11},2)$.

\medskip
In searching for exceptional cases of $m\,$---``exceptional'' 
in the sense that the expected equality
$\TC^S(\P^m)=\emb(\P^m)$ 
fails---, the situations to consider are:

\medskip\noindent{\bf Case $m=3$:} This was the first
undecided situation before this paper. The question was~(\cite{symmotion}):
\begin{equation}\label{question1}
\mbox{\em Is \ }\TC^S(\P^3)=4, \mbox{\em \ or is \ }\TC^S(\P^3)=5?
\end{equation}
As shown in Section~\ref{secexa},
the method in Project~\ref{problema1} allows one 
to show that the second possibility in~(\ref{question1})
is the correct answer. Remark~\ref{suerte} suggests that it might 
be possible to resolve the cases $m=5$ and $m=6$ discussed below 
along the same lines.

\medskip\noindent{\bf Case $m=5$:} 
The focus is on answering the question:
\begin{equation}\label{question2}
\mbox{\em Is \ }\TC^S(\P^5)=8, \mbox{\em \ or is \ }\TC^S(\P^5)=9?
\end{equation}
But there is an interesting subtlety not present in~(\ref{question1}).
Start by observing that the ($n=8$)-Taylor tower~(\ref{tower}) for $\P^5$ is 
analytic and has a nonempty\footnote{Since $\imm=\TC\leq\TC^S$, the nonempty
starting space of immersions is also the case in all the remaining 
situations of the section.} 
starting space $\imm(\P^5,\mathbb{R}^8)$, 
but that there must be a nontrivial obstruction
since the embedding dimension of $\P^5$ is known to be 9. 
Moreover,~(\ref{equivalencerange}) implies that any such obstruction 
has to show up when trying to lift through $r_2$, $r_3$, or $r_4$.
($\eta_4$ is an equivalence in the present situation). 
But then a potential (exceptional) case 
$\TC^S(\P^5)=8$ does not rule out a possible scenario where
the obstruction arises right at the very first lifting $r_2$; 
it would just mean that, in such a hypothetical situation, the 
space $\map^{\mathbb{Z}/2}(F(\P^5,2),S^7)$ wouldn't be
the right model for ${\cal T}_2{\cal E}_8^{\P^5}$. These possibilities 
have been explained in some detail since, from a different viewpoint, they 
could be used to actually produce new embeddings of higher-dimensional 
projective spaces---as discussed in the next cases.

\medskip\noindent{\bf Case $m=6$:} 
Here the focus is on answering~(\ref{question2}) with $\P^6$ replacing $\P^5$.
Note that the exceptional situation with $\TC^S(\P^6)=8$ is not quite
similar to that for $m=5\,$---there is a possible lack of analyticity now. 
However, there is a new interesting point coming from the fact that 
the explicit value of $\emb(\P^6)$ is currently unknown. Namely, 
one could try to use 
the approach in Project~\ref{problema2} in order to construct new
embeddings of $\P^6$. For instance, the simplest case would be to take 
$M=\P^6$ and $n=10$ in~(\ref{tower})---trying to produce the 
(new\footnote{This could be interpreted as smoothing Rees' 
PL-embedding of $\P^6$ into $\mathbb{R}^{10}$~(\cite{rees}).}) embedding
\begin{equation}\label{new1}
\P^6\subset\mathbb{R}^{10}. 
\end{equation}
Although this situation is still not within Haefliger's 
metastable range, the known fact $\TC^S(\P^6)\le9$ seems to suggest that
there shouldn't be any obstruction to lifting a given immersion 
$g\colon\P^6\immto\mathbb{R}^{10}$ through
the corresponding tower~(\ref{tower}) to produce an element
$\widetilde{g}\in{\cal T}_2{\cal E}_{10}^{\P^6}(\P^6)$. 
But it is even more interesting to note that since
$\eta_3$ is an equivalence in our current range, there is just
one obstruction for lifting $\widetilde{g}$ to the desired 
embedding~(\ref{new1}): Munson's
(secondary) obstruction. So (potential profit!), 
if there is no such secondary obstruction, then
one would be left with the (previously unknown smooth) 
embedding~(\ref{new1}).

\medskip\noindent{\bf Case $m=7$:} 
This situation is entirely similar to the one in the previous case, 
with Project~\ref{problema2} being a potential 
tool for producing new embeddings of $\P^7$ (a possible embedding into 
$\mathbb{R}^{11}$, smoothing Rees', is now the new simplest case to try).
But in this case there is one further intriguing feature,
namely, the possibility of using 
the approach in Project~\ref{problema1} to prove $\TC^S(\P^7)>9$; this 
would imply the new nonembedding result $\P^7\not\subset\mathbb{R}^{9}$
(such a possible result would be `strong' in view of the 
parallelizability of $\P^7$).

\medskip\noindent{\bf Cases with $m\ge11$:} 
In all these cases 
$\eta_3$ is an equivalence in the relevant range, so the considerations 
about using Project~\ref{problema2} as discussed around~(\ref{new1}) 
apply here too (i.e.~Munson's secondary obstruction could play an
important role in constructing new embeddings for these $\P^m$).

\smallskip
Note that if one could settle the relation $\TC^S(\P^{11})>17$---following, 
say, the guidelines in Project~\ref{problema1}---, then not only 
the actual value of $\TC^S(\P^{11})$ would be settled, 
but it would also follow that $\emb(\P^{11})=18\,$---again, this would be 
a new result. Similar considerations apply to
$\P^{13}$ and $\P^{14}$, as well as to $\P^{12}$ if one could 
prove---the rather unlikely---$\TC^S(\P^{12})>20$. But it is interesting 
to note that a new embedding result, this time for $\P^{12}$, would 
also follow if 
one could actually prove $\TC^S(\P^{12})\leq20$; indeed, such an inequality
would produce the new embedding $\P^{12}\subset\mathbb{R}^{20}$ in
view of~(\ref{metarange}) and~(\ref{tcsenelmeta}).

\smallskip
But perhaps one of the most fruitful cases to consider is that of $\P^{15}$,
where functoriality properties could be exploited in an eventual 
analysis of obstructions to lift elements in Taylor towers. For instance,
in the exceptional case that $\TC^S(\P^{15})=22$, any actual
element $x\in{\cal T}_2{\cal E}_{22}^{\P^{15}}(\P^{15})$ would necessarily have
a nontrivial Munson's obstruction (because $\emb(\P^{15})>22$). 
Now, if such an obstruction were to depend only on the $14$-th (resp.~$13$-th)
skeleton of $\P^{15}$ (say by an explicit calculation), 
then the corresponding (restricted) element $\bar{x}\in
{\cal T}_2{\cal E}_{22}^{\P^{14}}(\P^{14})$
(resp.~${\cal T}_2{\cal E}_{22}^{\P^{13}}(\P^{13})$)
would also have a nontrivial Munson secondary obstruction to lift to
${\cal T}_3{\cal E}_{22}^{\P^{14}}(\P^{14})$ 
(resp.~${\cal T}_3{\cal E}_{22}^{\P^{13}}(\P^{13})$).
The point then is that, modulo usual primary and secondary indeterminacy 
considerations, this could lead to the (again new, but now optimal) 
potential nonembedding result $\P^{14}\not\subset\mathbb{R}^{22}$ 
(respectively $\P^{13}\not\subset\mathbb{R}^{22}$ and 
$\P^{14}\not\subset\mathbb{R}^{22}$).

\smallskip
Finally, and although the following remark has been noted in a general form 
in the first paragraph of the case $m\ge11$, let us observe that,
in view of Rees' PL embedding $\P^{15}\subset\mathbb{R}^{23}$~(\cite{rees}),
and since at any rate $\TC^S(\P^{15})\le23$, 
Munson's secondary obstruction for an element in 
${\cal T}_2{\cal E}_{23}^{\P^{15}}(\P^{15})$ is the only obstruction
to producing a potential embedding $\P^{15}\subset\mathbb{R}^{23}
\,$---again,
this would be a new result, optimal in fact, that would smooth
Rees' PL embedding.

\section{Mod 2 cohomology of $B(\P^m,2)$}\label{redsymmcalc}
This section starts with a description of Haefliger's method~\cite{haefliger} 
for computing the cohomology ring $H^*(B(M,2);\F2)$ for
a closed smooth $m$-dimensional manifold $M$ (see~\cite{FGsymm}).
This information is then analyzed for $M=\P^m$ 
in connection with some of the lower bounds 
for $\TC^S(\P^m)$ in Table~\ref{tabla}.
Unless otherwise specified, throughout this section $H^*(X)$ will stand for
cohomology groups (or algebra, depending on the context)
where coefficients are taken 
in $\F2$, the field with
2 elements. The notation $\Z2$ or $\mathbb{Z}/2$ 
will also be used when referring to the group structure in $\F2$.

\medskip
Start with the Borel construction 
$S^\infty\times_{\Z2}M^2$, the total 
space in the standard fibration
\begin{equation}\label{borelfib}
M^2\to S^\infty\times_{\Z2}M^2\to\P^\infty
\end{equation}
where $\Z2$ acts on $M^2$ by swapping factors.
Steenrod showed (see~\cite[Subsection~2.4]{haefliger}
for a sketch of a proof of this and the 
subsequent facts in this paragraph)
that the Serre spectral sequence for this 
fibration collapses, so that one gets a 
ring isomorphism
\begin{equation}\label{collapse}
H^*(S^\infty\times_{\Z2}M^2)\cong 
H^*(\P^\infty;H^*(M)^{\otimes2}).
\end{equation}
Here $\pi_1(\P^\infty)=\Z2$ acts on 
$H^*(M)^{\otimes2}$ by swapping factors.
In particular the action is trivial 
on {\em diagonal} elements $x\otimes x$,
whereas $x\otimes y$ and $y\otimes x$ generate a split free
$\Z2$-submodule if $x\neq y$ and $x\neq0\neq y$. 
Thus, fixing\footnote{The author thanks 
Peter Landweber for indicating the fact that the correct definition of $D$
in~\cite{haefliger} should be given in terms of a basis of $H^*(M)$, and 
by noticing that the resulting $N$ is independent of the chosen basis.} 
a basis $\{a_r\}$ of $H^*(M)$,~(\ref{collapse})
transforms into the ring isomorphism
\begin{equation}\label{DN}
H^*(S^\infty\times_{\Z2}M^2)\cong 
\left(\F2[z]\otimes D\right)\oplus N
\end{equation}
where $z$ is the image of the generator $z\in H^1(\P^\infty)$
under the projection map in~(\ref{borelfib}), $D$ is additively
generated by the diagonal elements $1\otimes a_r\otimes a_r$
(also denoted by $a_r\otimes a_r$), and $N$ is additively
generated by the $\Z2$-invariant sums $a_r\otimes a_s+a_s\otimes a_r$ 
(with $r\neq s$). Note that~(\ref{DN})
is an isomorphism of 
($H^*(\P^\infty)=\F2[z]$)-algebras, where $z$
acts trivially on $N$ but freely on $D$. 
In particular, the product of an
element $z^i\otimes a_r\otimes a_r\in\F2[z]\otimes D$ and an 
element $a_s\otimes a_t+a_t\otimes a_s\in N$ is trivial for $i>0$, whereas  
$(a_r\otimes a_r)(a_s\otimes a_t+a_t\otimes a_s)=
a_ra_s\otimes a_ra_t+a_ra_t\otimes a_ra_s$ is easily seen to lie 
in $N$ (in particular $N$ is a subring---but $D$ is not). Furthermore,
the Steenrod algebra action on $H^*(M)$
is closely related to the  
diagonal map $\Delta\colon\P^\infty\times 
M=S^\infty\times_{\Sigma_2}M\to S^\infty\times_{\Sigma_2}M^2$,
where $\Z2$ acts trivially on $M$. 
Indeed, $\Delta^*$ is an $\F2[z]$-algebra 
map vanishing on $N$ such that
\begin{equation}\label{cohodelta}
\Delta^*(x\otimes x)=1\otimes x^2+
z\otimes\Sq^{k-1}x+z^2\otimes\Sq^{k-2}x+
\cdots+z^k\otimes x
\end{equation}
for $x\in H^k(M)$ (see for instance~\cite[page~500]{hatcher}).

\smallskip
To determine the multiplicative structure in $H^*(B(M,2))$,
Haefliger considers the map induced in cohomology by the inclusion 
$$j\colon B(M,2)\simeq S^\infty\times_{\Z2}
F(M,2)\hookrightarrow S^\infty\times_{\Z2}M^2.$$ 
Of course $j^*$ is a ring morphism, but it
turns out to be surjective. Therefore the 
multiplicative structure in $H^*(B(M,2))$ will be determined
from that of~(\ref{DN}) once $\ker j^*$ is described---in 
Theorem~\ref{haefligerstheorem} below.
First a little notation. Consider the push-forward map
\begin{equation}\label{pushforward}
\Delta_!\colon H^{*-m}(M)\to H^*(M^2)
\end{equation}
induced by the diagonal embedding $\mathrm{diag}
\colon M\hookrightarrow M\times M$. 
This is given by 
\begin{equation}\label{sorpresa}
\Delta_!(x)=(1\otimes x)\smallsmile\delta=(x\otimes 1)\smallsmile\delta,
\end{equation}
where $\delta\in H^m(M^2)$ is the diagonal 
cohomology class---the image of the fundamental class
under the restriction map
$H^m(M\times M,M\times M-\mathrm{diag})\to H^m(M\times M)$.
The final piece of information Haefliger needs is given by
the (degree $m$) endomorphism 
$\varphi\colon H^{*}(\P^\infty\times M)\to
H^{*}(\P^\infty\times M)$ given by multiplication
by the class
\begin{equation}\label{multiplication}
z^{m}\otimes w_{0}+z^{m-1}\otimes w_{1}+\cdots+
1\otimes w_{m},
\end{equation}
where $W=W_{M}=\sum_{i\ge0}^mw_{i}$ is the total 
Stiefel-Whitney class of $M$.

\begin{teorema}[Haefliger \cite{haefliger}]
\label{haefligerstheorem}
There is a commutative diagram with exact rows

\medskip\centerline{\xymatrix{0 \ar[r] & 
H^{*-m}(M)\ar[r]^{\Delta_!} 
& H^*(M^2) & & \\ 0 \ar[r] & H^{*-m}(\P^\infty\times M) 
\ar[u]^{r_1} \ar[r]^{\mu} \ar@{=}[d] & 
H^*(S^\infty\times_{\Z2}M^2) \ar[u]_{r_2} \ar[d]^{\Delta^*} 
\ar[r]^>>>>>{j^*} & H^*(B(M,2)) \ar[r] & 0 \\ 0 \ar[r] & 
H^{*-m}(\P^\infty\times M) \ar[r]^{\varphi} & 
H^*(\P^\infty\times M) & & }}

\medskip\noindent
Here $r_2$ is induced by the fiber inclusion in~{\em(\ref{borelfib})};
the situation for $r_1$ is similar (in terms of the 
identification $\P^\infty\times M=S^\infty\times_{\Z2}M$).
\end{teorema}

\begin{nota}\label{steenrodND}{\em
As an easy consequence of the fact that 
$r_2$ is monic on $N=\ker\Delta^*$, one finds that
every Steenrod square $\Sq^i\colon
H^*(S^\infty\times_{\Z2}M^2)\to H^*(S^\infty\times_{\Z2}M^2)$ 
satisfies $\Sq^i(n)\in N$ for $n\in N$. Thus, $\Sq^i(n)$ can 
be computed directly in $H^*(M^2)$---with the Cartan formula.
On the other hand, according to~\cite[Lemma~11]{bausum} 
(see also~\cite[Section VI]{yo}), an element $a\otimes a\in D$
has $\Sq^1(a\otimes a)=(\Sq^1a)\otimes a+
a\otimes(\Sq^1a)$ only when $\deg(a)$ is even, otherwise
$\Sq^1(a\otimes a)=z\otimes a\otimes a+(\Sq^1a)\otimes a+
a\otimes(\Sq^1a)$.
}\end{nota}

A simple diagram chase gives:

\begin{corolario}\label{kernel}
Let $a\in H^*(S^\infty\times_{\Z2}M^2)$. The following
conditions are equivalent:
\begin{itemize}
\item[{\em 1.}] $j^*(a)=0$.
\item[{\em 2.}] $\Delta^*(a)=\varphi(b)$ and $\,r_2(a)=\Delta_!\circ 
r_1(b)$, $\,$for some $\,b\in H^{*-m}(\P^\infty\times M)$.
\end{itemize}
\end{corolario}

Since $\varphi$ is monic, the term $b$ in
Corollary~\ref{kernel} is unique; it is in fact the preimage 
of $a$ under the monomorphism 
$\mu\colon H^{*-m}(\P^\infty\times M)\to H^*(S^\infty\times_{\Z2}M^2)$---whose
complete image is of course the required $\ker j^*$.

\begin{nota}\label{structure}{\em
For calculations it is convenient to observe that
$\mu$ is a map of $\F2[z]$-modules---this
follows from the behavior of $\Delta^*$, and 
the described multiplicative structure in~(\ref{DN}).
}\end{nota}

Thus, concrete numerical calculations for the ring 
structure in $H^*(B(M,2))$ require knowledge of
the Stiefel-Whitney classes $w_i$, the action of the 
Steenrod algebra on $M$,
and the diagonal class $\delta$ associated to $M$. 
For the latter, it is convenient
to keep in mind the following characterization.

\begin{teorema}[Theorem~11.11 in~\cite{MR0440554}]
\label{diagonalclass}
Fix a basis $b_1,\ldots,b_r$ of $H^*(M)$, and let
$b'_1,\ldots,b'_r$ stand for the corresponding dual basis. Then
$\delta=\sum_{i=1}^{r}b_i\otimes b'_i.$
\end{teorema}

Next, Haefliger's analysis will be specialized to the case $M=\P^m$
(the author does not know if this has been done elsewhere; however, 
the calculation implicit from~\cite[Theorem~3.7]{handel} should be noted). 
Start from~(\ref{DN}) noticing that $D=\F2[\lambda]/\lambda^{m+1}$
where $\lambda=z\otimes z$. Here $z$ stands for the restriction 
to $\P^m$ of the generator $z\in H^1(\P^\infty)$. The reader should keep 
in mind that this notation has a different use than that intended 
in~(\ref{DN}), but the context clarifies any possible confusion. For instance,
an additive $\F2$-basis for $N$ is 
given by the monomials 
\begin{equation}\label{addbas}
z^i\otimes z^k+z^k\otimes z^i,\quad 0\leq i<k\leq m,
\end{equation}
whereas an $\F2$-basis for the first summand on the right-hand-side 
of~(\ref{DN}) is given by the elements $z^i\otimes\lambda^j=
z^i\otimes z^j\otimes z^j$, with $i,j\geq0$ and $j\leq m$.

\medskip
Consider the polynomial expressions 
$Q_i=Q_i(\lambda,\eta)\in N$ defined by the relation
\begin{equation}\label{eiqi}
\eta^i+Q_i=1\otimes z^i+z^i\otimes1,\quad i\geq1,
\end{equation}
where $\eta=1\otimes z+z\otimes1$,
so that the basis in~(\ref{addbas}) takes the form 
$\lambda^i(\eta^{k-i}+Q_{k-i})$. An upper triangular matrix 
then changes~(\ref{addbas}) to the basis $\lambda^i\eta^{k-i}$ for $N$.
This shows 
\begin{equation}\label{Borelproduct}
H^*(S^\infty\times_{\Z2}(\P^m)^2)\approx\left.\rule{0mm}{3.8mm}
\F2[\zeta,\lambda,\eta]\right/(\zeta\eta,R_0,\ldots,
R_{m+1})
\end{equation}
where $\zeta=z\otimes1\otimes1$ and 
$R_i=\lambda^{m+1-i}(\eta^i+Q_i)$ (setting $Q_0=0$).
Theorem~\ref{haefligerstheorem} and Remark~\ref{structure}
then yield the ring isomorphism
\begin{equation}\label{mo2cohoB}
H^*(B(\P^m,2))\approx\left.\rule{0mm}{3.8mm}
\F2[\zeta,\lambda,\eta]\,\right/I_m
\end{equation}
where $I_m$ is the ideal 
generated by $\zeta\eta$, the $R_i$ ($0\leq i\leq m+1$),
and the $\mu(1\otimes z^k)$ ($0\leq k\leq m$). 

\begin{nota}\label{gennotmin}{\em
The given set of generators for $I_m$
is not minimal: take $m=2$, then
(\ref{primera})--(\ref{tercera}) below show that $\mu(1\otimes z^2)$ and
all the generators $R_i$ are redundant, so that
$I_2=(\zeta\eta,\lambda+\zeta^2+\eta^2,\zeta\lambda+\eta\lambda)$.
This leads to~(\ref{cohop2}) below, after eliminating the variable $\lambda$.
}\end{nota}

In order to make~(\ref{mo2cohoB}) into
an explicit expression, one would need to know the
$Q_i$ and the $\mu(1\otimes z^k)$ as polynomials in the variables
$\zeta,\lambda,\eta$. The former set of polynomials
depends only on $\lambda$ and 
$\eta$ and, in fact,~(\ref{eiqi}) can be used to get the inductive formula
\begin{equation}\label{Qi}
Q_i=\sum^{\left[\frac{i-1}{2}\right]}_{k\geq1}\binom{i}{k}\lambda^k\left(
\rule{0mm}{3.8mm}\eta^{i-2k}+Q_{i-2k}\right).
\end{equation}
Here $\left[\frac{i-1}{2}\right]$ stands for the integral part of $(i-1)/2$.
For instance: $Q_0=Q_1=Q_2=Q_4=Q_8=0$, $\,Q_3=\lambda\eta$, $\,Q_5
=\lambda\eta^3+\lambda^2\eta$, $\,Q_6=\lambda^2\eta^2$, and $Q_7
=\lambda\eta^5+\lambda^3\eta$.
%\begin{align}
%Q_0&=Q_1=Q_2=Q_4=Q_8=Q_{16}=0;\nonumber\\
%Q_3&=\lambda\eta;\nonumber\\
%Q_5&=\lambda\eta^3+\lambda^2\eta;\nonumber\\
%Q_6&=\lambda^2\eta^2;\nonumber\\
%Q_7&=\lambda\eta^5+\lambda^3\eta;\nonumber\\
%Q_9&=\lambda\eta^7+\lambda^2\eta^5+\lambda^4\eta;\nonumber\\
%Q_{10}&=\lambda^2\eta^6+\lambda^4\eta^2;\nonumber\\
%Q_{11}&=\lambda\eta^9+\lambda^3\eta^5+\lambda^4\eta^3+\lambda^5\eta;\nonumber\\
%Q_{12}&=\lambda^4\eta^4\nonumber\\
%Q_{13}&=\lambda\eta^{11}+\lambda^2\eta^9+\lambda^5\eta^3+\lambda^6\eta
%;\nonumber\\
%Q_{14}&=\lambda^2\eta^{10}+\lambda^6\eta^2;\nonumber\\
%Q_{15}&=\lambda\eta^{13}+\lambda^3\eta^9+\lambda^7\eta;\nonumber
%\end{align}
As for the polynomials $\mu(1\otimes z^k)$, Theorem~\ref{haefligerstheorem}
can be used, in principle, to get non-inductive expressions 
for these elements as soon as one knows the three 
maps $\Delta_{!}$, $\Delta^*$, and $\varphi$ in the case $M=\P^m$ 
(expressions for $r_1$ and $r_2$ are simple; the latter, for instance,
is the identity on $D$ and $N$, but vanishes on any $z$-multiple). $\Delta_{!}$
is determined by~(\ref{sorpresa}) and Theorem~\ref{diagonalclass}
which yields 
\begin{equation}\label{deltachica}
\delta=z^m\otimes1+z^{m-1}\otimes z+\cdots+z\otimes z^{m-1}+1\otimes z^m.
\end{equation}
Expression~(\ref{multiplication}) can be written down
in a very compact form: since $W_{\P^m}=(1+z)^{m+1}$ and multiplication
by $z\otimes 1$ is injective in $H^*(\P^\infty\times\P^m)$, $\varphi$ 
is multiplication by 
\begin{equation}\label{compact}
\frac{\;\left(1\otimes z+z\otimes1\right)^{m+1}}{z\otimes1}
\end{equation}
---which is well defined since $0=1\otimes z^{m+1}\in H^*(\P^\infty\times\P^m)$.
Likewise,~(\ref{cohodelta}) takes the compact form
\begin{equation}\label{squares}
\Delta^*(z^k\otimes z^k)=\left(1\otimes z^k\right)
\left(1\otimes z+z\otimes1\right)^k.
\end{equation}

Although the above information suffices to perform explicit computations,
details soon get combinatorially complex as $m$ increases. Thus, 
after the following technical lemma (needed in connection with 
Project~\ref{problema1}), only a few 
complete examples will be analyzed (for $m\leq3$). In addition,
the final part of this section offers a description of the 
height of the first Stiefel-Whitney class of the bundle 
$\zeta_m$ in Section~\ref{secprojuno} for some families of $m$, 
and its relation, in terms of Project~\ref{problema1}, 
to the lower bounds in Table~\ref{tabla}.

\begin{lema}[Compare to~{\cite[page~278]{bausum}}]\label{classifying}
For a closed smooth $m$-dimensional manifold $M$,
the classifying map $B(M,2)\to\P^\infty$ for the principal $\Z2$-bundle 
$F(M,2)\to B(M,2)$
corresponds to the $\F2$-cohomology class $\zeta=j^*(z\otimes1\otimes1)$.
\end{lema}
\begin{proof}
This follows from the commutative diagram 

\begin{picture}(0,50)(-36,-5)
\put(0,32){$F(M,2)$}
\put(53,32){$S^\infty\times F(M,2)$}
\put(133,32){$S^\infty\times M^2$}
\put(195,32){$S^\infty$}
\put(0,0){$B(M,2)$}
\put(49,0){$S^\infty\times_{\Z2} F(M,2)$}
\put(128.7,0){$S^\infty\times_{\Z2} M^2$}
\put(195,0){$\P^\infty$}
\put(13,27){\vector(0,-1){15}}
\put(76,27){\vector(0,-1){15}}
\put(150,27){\vector(0,-1){15}}
\put(198,27){\vector(0,-1){15}}
\put(49,34){\vector(-1,0){17}}
\put(107.5,34){\vector(1,0){22}}
\put(170,34){\vector(1,0){22}}
\put(46,2){\vector(-1,0){14}}
\put(110,2){\vector(1,0){16}}
\put(173,2){\vector(1,0){18}}
\end{picture}

\noindent where left-hand-side horizontal maps are homotopy equivalences, 
middle horizontal maps are inclusions, and the right-hand-side square
is a pull-back with horizontal maps projecting to the first 
coordinate.
\end{proof}

\begin{ejemplos}\label{ejemp2}{\em
Take $M=\P^2$, the projective plane,
so that $W_{\P^2}=1+z+z^2$. Thus both~(\ref{deltachica}) 
and~(\ref{compact}) are given by $z^2\otimes 1+z\otimes z+1\otimes z^2$.
A direct calculation using~(\ref{squares}) then shows that the 
$\F2[z]$-monomorphism $\mu$
in Theorem~\ref{haefligerstheorem} is determined by
\begin{align}
\mu(1\otimes1)\,\,&=\;
z^2\otimes1\otimes1+1\otimes z\otimes z+(z^2\otimes1+1\otimes z^2)\,\;=\;\,
\zeta^2+\lambda+\eta^2;
\label{primera}\\
\mu(1\otimes z)\,\,&=\;
z\otimes z\otimes z+(z^2\otimes z+z\otimes z^2)\,\;=\;\,
\zeta\lambda+\lambda\eta; \label{segunda}\\
\mu(1\otimes z^2)& =\; 1\otimes z^2\otimes z^2\;\,=\;\,
\lambda^2; \label{tercera}
\end{align}
(observe that $\dim(\lambda)=2$ and $\dim(\zeta)=\dim(\eta)=1$) 
and, after a little algebraic manipulation,~(\ref{mo2cohoB}) becomes
\begin{equation}\label{cohop2}
H^*(B(\P^2,2))=\left.\rule{0mm}{3.8mm}
\F2[\zeta,\eta]\,\right/(\zeta\eta,\zeta^3+\eta^3).
\end{equation}
In a similar manner one derives
$$H^*(B(\P^3,2))=\left.\rule{0mm}{3.8mm}
\F2[\zeta,\lambda,\eta]\,\right/(\zeta\eta,
\lambda^3,\zeta^3+\eta^3,\zeta\lambda^2+\lambda^2\eta,
\zeta^2\lambda+\lambda^2+\lambda\eta^2).$$
Table~\ref{T2} gives an explicit additive $\F2$-basis for this algebra,
whereas Remark~\ref{steenrodND} gives the formulas $\Sq^1(\zeta)=\zeta^2$,
$\Sq^1(\eta)=\eta^2$, and $\Sq^1(\lambda)=\lambda(\zeta+\eta)$ (the first
two are forced by dimensional reasons). In Section~\ref{secexa} we will
need to use some information about 
the $\Sq^1$-cohomology of $H^*(B(\P^3,2))$. A straightforward calculation
shows this to be given by:
\begin{itemize}
\item $\mathbb{Z}_2$, in dimensions $0$ and $4$ (represented by $1$ and 
$\lambda(\zeta^2+\eta^2)$, respectively);
\item $\mathbb{Z}_2\oplus\mathbb{Z}_2$, in dimension $3$ (represented by 
$\zeta^3$ and $\lambda\zeta$);
\item $0$, in any other dimension.
\end{itemize}
\begin{table}[h]
\centerline{
\begin{tabular}{|c|c|c|c|c|c|c|}\hline
{\it basis}$\rule{0mm}{4mm}$&$1$&$\zeta$, $\eta$&$\zeta^2$, $\eta^2$,
$\lambda$&$\zeta^3$, $\lambda\zeta$, $\lambda\eta$&$\lambda\zeta^2$, 
$\lambda\eta^2$&$\lambda\zeta^3$\\ \hline
{\it dimension}$\rule{0mm}{4mm}$&$0$&$1$&$2$&$3$&$4$&$5$\\ \hline
\end{tabular}}
\caption{Basis elements in $H^*(B(\P^3,2))$\label{T2}}
\end{table}
}\end{ejemplos}

\begin{nota}\label{cruzada}{\em
The relations $\eta^4=0$ and $\zeta^4=0$ clearly hold in~(\ref{cohop2}),
but neither $\eta^3$ nor $\zeta^3$ vanishes. Of particular interest is 
the non-triviality of the last element since it implies, 
from Lemma~\ref{classifying}, that the classifying map for $\zeta_2$
cannot be deformed into a map $B(\P^2,2)\to\P^2$, and therefore, 
as described in Section~\ref{secprojuno},  
$\TC^S(\P^2)\ge4$ (which is in fact an equality, as indicated 
in Table~{\ref{tabla}}). This approach can be tried for larger-dimensional
projective spaces (details below), but the lower bounds 
thus obtained do not improve on (but, for $m\leq8$, coincide with) those 
in~\cite{symmotion}. (The situation is comparable with that observed 
in the first complete paragraph in page 126 of~\cite{handel}.) For instance, 
although $\zeta^4=0$ is clearly a  relation in $H^*(B(\P^3,2))$,
$\zeta^3$ does not vanish in this ring (its restriction to $H^*(B(\P^2,2))$
is nontrivial). Thus $\TC^S(\P^3)\ge4$ is all one can 
deduce from this $\F2$-approach. But much of the motivation 
for Project~\ref{problema1} comes from the fact that, 
by replacing $\F2$-cohomology with integral cohomology, the above ideas 
allow us to get, in Section~\ref{secexa},
the improved $\TC^S(\P^3)\geq5$, a sharp  
inequality in view of Remark~\ref{obstr1} and the known $\emb(\P^3)=5$.
(Remark~\ref{razon} pinpoints the reason why the $\F2$-approach fails.)
}\end{nota}

The rest of this section is devoted to describing the height of 
$\zeta\in H^*(B(\P^m,2))$ for some families of values of $m$, and to indicating 
the way this compares to the lower bounds in Table~\ref{tabla}.

\medskip\noindent {\bf Case $m=1$:} One gets
$R_0=\lambda^2$, $R_1=\lambda\eta$, $R_2=\eta^2$, $\mu(1\otimes1)=\zeta+\eta$,
and $\mu(z\otimes z)=\lambda$. Thus $I_1$ reduces to the ideal generated
by $\zeta\eta$, $\eta^2$, $\zeta+\eta$, and $\lambda$, and $H^*(B(\P^1,2))
\approx\F2[\zeta]/\zeta^2$ which, of course, is compatible with the fact that
$B(\P^1,2)\simeq S^1$. Under these conditions, 
the Hopf-type $\F2$-approach in Section~\ref{secprojuno} gives
$\TC^S(\P^1)\geq2$---optimal in view of Table~\ref{tabla}.

\medskip\noindent {\bf Case $m=2^e$:} As a partial generalization of 
the previous case, it is now affirmed that $0\neq\zeta^{2^{e+1}-1}
\in H^*(B(\P^{2^e},2))$, so that the Hopf-type $\F2$-approach in 
Section~\ref{secprojuno} gives $\TC^S(\P^{2^e})\geq2^{e+1}$---which is optimal
for $e\leq3$ in view of Table~\ref{tabla},
and for $e\geq4$ in view of~\cite{dontables,symmotion}. Indeed, in the notation
of Theorem~\ref{haefligerstheorem} one has $\Delta^*(z^{2^{e+1}-1}\otimes1
\otimes1)=z^{2^{e+1}-1}\otimes1$. But an easy calculation shows that
the preimage of this element under $\varphi$ is 
$$ %\begin{equation}\label{preimage}
z^{2^e-1}\otimes1+z^{2^e-2}\otimes z+\cdots+z\otimes z^{2^e-2}+1\otimes z^{2^e-1}.
$$ %\end{equation}
However, the last element maps nontrivially under $r_1$,
while $r_2(z^{2^{e+1}-1}\otimes1\otimes1)=0$.

\medskip\noindent {\bf Case $m=2^e+\varepsilon$ with $\varepsilon\in\{1,2\}$
and $e\geq2$:} 
The previous analysis implies $0\neq\zeta^{2^{e+1}-1}
\in H^*(B(\P^{2^e+\varepsilon},2))$. It is now affirmed that $0=\zeta^{2^{e+1}}
\in H^*(B(\P^{2^e+\varepsilon},2))$, so that the best information one gets from
the Hopf-type $\F2$-approach in 
Section~\ref{secprojuno} is $\TC^S(\P^{2^e+\varepsilon})\geq2^{e+1}$. 
Indeed, it is enough to consider the case $\varepsilon=2$,
where a new calculation gives that the $\varphi$-preimage of
$\Delta^*(z^{2^{e+1}}\otimes1\otimes1)=z^{2^{e+1}}\otimes1$ is 
$$
\sum^{2^{e-2}-1}_{k\geq0}\left(z^{2^e-4k-2}\otimes 
z^{4k}+z^{2^e-4k-3}\otimes z^{4k+1}\right).
$$
The last element maps trivially under $r_1$, so that it is in fact
the $\mu$-preimage of $z^{2^{e+1}}\otimes1\otimes1$, killing 
$\zeta^{2^{e+1}}\in H^*(B(\P^{2^e+\varepsilon},2))$.

\begin{nota}\label{suerte}{\em
Since $\TC^S(\P^a)\geq\TC^S(\P^b)$ is obvious for $a\geq b$,
the inequality $\TC^S(\P^{2^e+\varepsilon})\geq2^{e+1}$ (for $\varepsilon\in
\{1,2\}$) follows directly from the previously established 
$\TC^S(\P^{2^e})\geq2^{e+1}$. But unlike the latter, the former 
is actually {\em not\,} optimal for $e\geq3$: 
Table~\ref{tabla} gives $\TC^S(\P^9)=\TC^S(\P^{10})=17$, 
whereas $\TC^S(\P^{2^{e}+\varepsilon})=2^{e+1}+1$
for $e\geq4$ in view of~\cite{dontables,symmotion}.
In view of its success for $(e,\varepsilon)=(1,1)$,
there seems to be a ``good'' chance that
the ideas in Project~\ref{problema1} can be used to
settle the still unresolved cases with $e=2$ ($\varepsilon=1,2$). Of course, 
this would settle the value of $\TC^S(\P^m)$ for $m\le6$ and, on the other 
hand, it would show the equality $\TC^S(\P^m)=\emb(\P^m)$ for $m\leq 5$. 
}\end{nota}

\noindent {\bf Case $m=2^e+3$:} Just as in previous situations,
\begin{equation}\label{panconlomismo}
0=\zeta^{2^{e+1}}\in H^*(B(\P^{2^e+3},2)),
\end{equation}
so that the best information one gets from the Hopf-type $\F2$-approach in 
Section~\ref{secprojuno} is again 
$\TC^S(\P^{2^e+3})\geq2^{e+1}\,$---besides having one further
illustration of Haefliger's method, the 
reason for not including this case with the previous one is that
there is currently no clear evidence as to what
the actual value of $\TC^S(\P^{2^e+3})$ could be. 
To show~(\ref{panconlomismo}), this time 
one computes that the $\varphi$-preimage of
$\Delta^*(z^{2^{e+1}}\otimes1\otimes1)=z^{2^{e+1}}\otimes1$ is 
$$
\sum^{2^{e-2}-1}_{k\geq0}\left(z^{2^e-4k-3}\otimes z^{4k}\right).
$$
The last element maps trivially under $r_1$, so that it is in fact
the $\mu$-preimage of $z^{2^{e+1}}\otimes1\otimes1$, once again killing 
$\zeta^{2^{e+1}}\in H^*(B(\P^{2^e+3},2))$.

%\bigskip
%\centerline{\bf Further ideas for this section (se quitar\'a):}
%\
%\\noindent
%\So I guess the next step is: How to compute Bocksteins in an expression 
%\such as~(\ref{mo2cohoB})? Following a suggestion of Fred,
%\it certainly suffices to compute Bocksteins in~(\ref{DN}). He mentioned:
%\
%\\begin{quote}
%\To work out the Bocksteins, it suffices to know how to do this
%\for $W\times_G X^2$ which is implicit in  work of Browder.
%\A good way to think about this is by the usual basis for the
%\mod 2 homology of $W\times_G X^2$ and the known action of
%\the Bocksteins.
%\\end{quote}
%\
%\Fred also mentioned (among so many other ideas, which I still need
%\to check very carefully) taking a look at 
%\$$\amalg_n E\Sigma_n \times_{\Sigma_n}X^n \to Q(X \amalg +).$$

\section{Proofs of Theorems~\ref{vendimia} and~\ref{lacohofinal}}\label{secexa}
Unless otherwise noted, throughout this section $H^*(X)$ stands for 
the singular cohomology groups (or algebra, depending on the context) 
of a space $X$, where integral coefficients are used. 

\smallskip
It has been observed, at the end of Remark~\ref{cruzada}, that
in order to settle $\TC^S(\P^3)=5$, it is enough to
prove $\TC^S(\P^3)\geq5$. This inequality is established
in the present section within the setup in Project~\ref{problema1}. Indeed, 
as explained in Section~\ref{secprojuno},
the goal is to show that the classifying map for $\zeta_3$ cannot be
compressed into a map $$B(\P^3,2)\to\P^3.$$
Since $H^*(\P^\infty)=\mathbb{Z}[\omega]/2\omega$,
$\omega\in H^2(\P^\infty)$, with $\omega^2$ trivial on $\P^3$, the required 
conclusion can be stated as:

\begin{teorema}\label{notrivialidad}
$\omega^2$ maps non-trivially under 
the classifying map for $\zeta_3$. 
\end{teorema}

\begin{nota}\label{razon}{\em
Let $\widehat{\zeta_3}\colon B(\P^3,2)\to\P^\infty$ stand for the map in
Theorem~\ref{notrivialidad}.
It will become clear that $\widehat{\zeta^*_3}(\omega^2)$ is 2-divisible.
This is the reason why cohomology with mod 2 coefficients is not
useful for proving the required inequality $\TC^S(\P^3)\geq5$.
}\end{nota}

The proof of Theorem~\ref{notrivialidad}
is based on the Cartan-Leray-Serre spectral sequence 
(with integral coefficients)
for the $\Z2$-cover associated to $\zeta_3$. Much of this section is
devoted to giving full details of that spectral sequence. 

\smallskip
Recall that the fiber of 
$\widehat{\zeta}_3$ is the ordered configuration space $F(\P^3,2)$ with
($\Z2=\pi_1(\P^\infty)$)-action given by the involution 
$(u,v)\stackrel{t}\mapsto(v,u)$.
Thus, the spectral sequence to be used has 
\begin{equation}\label{CLSSS}
E_2^{p,q}=H^p(\P^\infty;H^q(F(\P^3,2)))\Longrightarrow H^{p+q}(B(\P^3,2)).
\end{equation}
Cohomology coefficients in this $E_2$-term 
are twisted by (the map induced by) $t$. 
A sound hold on~(\ref{CLSSS}) comes from the homeomorphism 
$\homeo\cong F(\P^3,2)$ given by $(a,b)\mapsto(a,ab)$ with inverse 
$(x,y)\mapsto(x,x^{-1}y)$, where $e$ is the identity matrix in
$\P^3=\mathrm{SO}(3)$, and inverses are taken with respect to the
group structure. In these terms, the resulting involution 
$\tau\colon\homeo\to\homeo$ takes the form $\tau(a,b)=(ab,b^{-1})$.

\smallskip
Recall $H^*(\P^3)=\mathbb{Z}[x,y]\,/\left(x^2,y^2,xy,2x\right)$
and $H^*(\fade)=\mathbb{Z}[x]\,/\left(x^2,2x\right)$, for cohomology
classes $x$ and $y$ of dimensions 2 and 3, respectively.
The K\"unneth isomorphism 
$H^*(X\times Y)\approx \left[H^*(X)\otimes H^*(Y)\right]^*\oplus
\left[\mathrm{Tor}\left(H^*(X),H^*(Y)\right)\right]^{*+1}$
yields:

\begin{lema}\label{cohoZ}
$H^*(\homeo)$ is the direct sum of 
$$\mathbb{Z}[x_1,y_1,x_2]\,\left/\left(x_1^2,x_2^2,y_1^2,x_1y_1,
2x_1,2x_2\right)\right.$$ and a copy of $\,\mathbb{Z}/2$ generated 
by a class $z\in H^3(\homeo)$.
\end{lema}

The index $i$ in $x_i$ and $y_i$ refers to the Cartesian factor where 
the indicated classes originate. The only class coming from the Tor part,
the class $z$,
arises from the two groups $H^2(\P^3)\approx H^2(\fade)\approx\mathbb{Z}/2$.
The whole multiplicative structure in $H^*(\homeo)$ 
is determined by specifying 
the four products $x_1z$, $x_2z$, $y_1z$, and $z^2$. Of these, the last 
two are trivial for dimensional reasons, whereas 
Example~\ref{elsegundo} below settles
the corresponding fact for $x_2z$. 
Although irrelevant for the calculations in this section, 
the author does not know whether 
$x_1z$ is trivial or not ($x_1z=y_1x_2$ would be forced in the latter case).

\medskip
The next step toward understanding~(\ref{CLSSS}) is to produce
a complete description of 
$\tau^*\colon H^*(\homeo)\to H^*(\homeo)$. 
This involution is easily seen to be trivial in the cases
$H^0(\homeo)=\mathbb{Z}$, $H^1(\homeo)=0$, $H^4(\homeo)\approx\mathbb{Z}/2$,
and $H^5(\homeo)\approx\mathbb{Z}/2$. The following result gives the answer 
in the two remaining cases.

\begin{proposicion}\label{twistedcoeffs}
In cohomology dimensions $2$ and $3$, the involution $\tau^*$ satisfies
\begin{equation}\label{cuatro}
\tau^*(x_1)=x_1+x_2,\quad\tau^*(x_2)=x_2,\quad\tau^*(y_1)=y_1+z,\quad 
\mbox{and}\quad\tau^*(z)=z.
\end{equation}
\end{proposicion}

\begin{ejemplo}\label{elsegundo}{\em
The relation $x_2z=0$ follows from $x_1z=\tau^*(x_1z)=\tau^*(x_1)\tau^*(z)=
(x_1+x_2)z=x_1z+x_2z$.
}\end{ejemplo}
\begin{proof}
The formula $\tau^*(z)=z$ follows from the observation that $z$ 
is the unique element in $H^3(\homeo)\approx\mathbb{Z}\oplus\mathbb{Z}/2$ 
of order two. For the remaining 
cases consider the diagram

\begin{picture}(0,62)(-17,-8)
\put(0,0){$\fade$}
\put(15,37){$\P^3$}
\put(50,20){$\homeo$}
\put(27,12){\vector(3,1){17}}
\put(27,34){\vector(3,-1){17}}
\put(112,22){\vector(1,0){20}}
\put(121,24){\scriptsize $\tau$}
\put(137,20){$\homeo$}
\put(198,17){\vector(3,-1){17}}
\put(198,28){\vector(3,1){17}}
\put(220,37){$\P^3$}
\put(213,0){$\fade$}
\end{picture}

\noindent where right-hand-side diagonal maps are Cartesian projections, and 
left-hand-side diagonal maps are 
Cartesian inclusions with respect to some chosen base point 
$d\in\fade\subset\P^3=\mathrm{SO}(3)$ of order two 
(e.g., $d=\mathrm{Diag}(-1,-1,1)$, so that $d=d^{-1}$). 
The four resulting components are depicted in

\begin{picture}(0,78)(-4,5)
\put(-5,70){$a$}
\put(20,40){$(a,d)$}
\put(2,65){\vector(1,-1){15}}
\put(43,42){\vector(1,0){12}}
\put(60,40){$(a\cdot d,d)$}
\put(109,70){$a\cdot d$}
\put(111,10){$d$}
\put(92,33){\vector(1,-1){15}}
\put(92,50){\vector(1,1){15}}
\put(145,10){$b$}
\put(152,18){\vector(1,1){15}}
\put(170,40){$(d,b)$}
\put(193,42){\vector(1,0){12}}
\put(210,40){$(d\cdot b,b^{-1})$}
\put(259,70){$d\cdot b$}
\put(261,10){$b^{-1}$}
\put(242,33){\vector(1,-1){15}}
\put(242,50){\vector(1,1){15}}
\end{picture}

\noindent Since $\P^3$ is a path-connected group, the components
$\P^3\to\P^3$ and $\fade\to\P^3$ are homotopic to inclusions,
whereas the component $\fade\to\fade$ is necessarily the identity in
$H^2(\fade)\approx\mathbb{Z}/2$. This yields the first two formulas 
in~(\ref{cuatro}). However, since $z$ is not detected by axial inclusions,
all one gets for the third formula in~(\ref{cuatro}) is $\tau^*(y_1)=y_1
+\epsilon z$, for some $\epsilon\in\{0,1\}$. In order to settle 
this indeterminacy, consider the portion 
\begin{eqnarray}
H^3(\homeo)\stackrel2\to H^3(\homeo)\stackrel{\mathrm{proj}}\longrightarrow
H^3(\homeo;\mathbb{Z}/2)\nonumber\\
\stackrel\partial\to H^4(\homeo)\stackrel2\to H^4(\homeo)
\quad\quad\quad\label{portion}
\end{eqnarray}
of the long exact sequence associated to the extension
$0\to\mathbb{Z}\stackrel2\to\mathbb{Z}\stackrel{\mathrm{proj}}\longrightarrow
\mathbb{Z}/2\to0$.
Since $H^3(\homeo)\approx\mathbb{Z}\oplus\mathbb{Z}/2$ and $H^4(\homeo)
\approx\mathbb{Z}/2$, the middle part in~(\ref{portion}) becomes 
$$0\to\mathbb{Z}/2\oplus\mathbb{Z}/2
\stackrel{\mathrm{proj}'}\longrightarrow\mathbb{Z}/2\oplus\mathbb{Z}/2\oplus
\mathbb{Z}/2\stackrel{\partial}\longrightarrow
\mathbb{Z}/2\to0,$$ 
so that the value of $\epsilon$ can be set 
by looking at $\tau^*\colon H^3(\homeo;\mathbb{Z}/2)\to
H^3(\homeo;\mathbb{Z}/2)$. Indeed, $\epsilon=1$ is forced from the 
commutative diagram

\begin{picture}(0,52)(0,4)
\put(84,40){$\P^3\times\P^3$}
\put(69,10){$\homeo$}
\put(172,40){$\P^3$}
\put(149,10){$\homeo$}
\put(120,42){\vector(1,0){45}}
\put(130,12){\vector(1,0){14}}
\put(175,21){\vector(0,1){14}}
\put(98,21){\vector(0,1){14}}
\put(135,15){\scriptsize $\tau$}
\put(140,45){\scriptsize $m$}
\put(178.5,26){\scriptsize $\pi$}
\end{picture}

\noindent 
(where the vertical map on the left-hand-side is the obvious inclusion,
and $m$ is the multiplication in $\P^3=\mathrm{SO}(3)$)
and the well-known fact that $m^*(g^3)=g^3\otimes1+g^2\otimes g+
g\otimes g^2+1\otimes g^3$, where $g$ is the generator in 
$H^1(\P^3;\mathbb{Z}/2)$.
\end{proof}

The $E_2$-term in~(\ref{CLSSS}) can now be obtained from standard 
calculations (e.g.~\cite[page~6]{coho}). The result, recorded next, 
is depicted in the chart following Remark~\ref{nmodule}.

\begin{corolario}\label{e2}
\begin{itemize}
\item[{\em 1.}] $H^*(\P^\infty;H^0(\homeo))=\mathbb{Z}[\omega]/2\omega, 
\quad\dim(\omega)=2$.
\item[{\em 2.}] $H^*(\P^\infty;H^1(\homeo))=0$.
\item[{\em 3.}] $H^*(\P^\infty;H^2(\homeo))=\begin{cases}
\mathbb{Z}/2, & *=0;\\0, & *>0.\end{cases}$
\item[{\em 4.}] $H^*(\P^\infty;H^3(\homeo))=\begin{cases}
\mathbb{Z}\oplus\mathbb{Z}/2, & *=0;\\
\mathbb{Z}/2, & \mathrm{positive\,\, even\,\,}*;\\
0, & \mathrm{odd\,\,}*.\end{cases}$
\item[{\em 5.}] $H^*(\P^\infty;H^q(\homeo))=\F2[\omega_q],\quad\dim(\omega_q)=1,
\quad q=4,5$.
\item[{\em 6.}] $H^*(\P^\infty;H^i(\homeo))=0,\quad i\geq6.$
\end{itemize}
\end{corolario}

\begin{nota}\label{nmodule}{\em
An explicit description of the 
$H^*(\P^\infty)$-module structure in the spectral sequence~(\ref{CLSSS}) 
will be crucial for getting a good control of differentials 
(in the next paragraphs). 
To begin with, as indicated in Corollary~\ref{e2}(1),
there is the standard copy of the ring 
$H^*(\P^\infty)$ at the ($q=0$)-line of~(\ref{CLSSS}), 
whereas Corollary~\ref{e2}(3) forces the $H^*(\P^\infty)$-module 
$H^*(\P^\infty;H^2(\homeo))$ at the ($q=2$)-line to 
have one generator (of dimension $p=0$) 
with both $2$ and $\omega$ acting trivially.
The situation at the lines $q=4$ and $q=5$ is well known; in the
notation of Corollary~\ref{e2}(5),
the $H^*(\P^\infty)$-module $H^*(\P^\infty;H^q(\homeo))$ is generated
by $1_q$ and $\omega_q$ subject to the single relation $2\cdot1_q=0$
(here $1_q$ stands for the unit of the ring in Corollary~\ref{e2}(5)).
Finally, in order to interpret Corollary~\ref{e2}(4), note first that
the last relation in~(\ref{cuatro}) claims that $\mathbb{Z}/2$ is a 
$\mathbb{Z}[\mathbb{Z}/2]$-submodule of $H^3(\homeo)=\mathbb{Z}
\oplus\mathbb{Z}/2$. Then, a standard calculation shows that the 
induced map $$H^*(\P^\infty;\mathbb{Z}/2)\to H^*(\P^\infty;H^3(\homeo))$$
is injective in even dimensions (and, therefore, an isomorphism 
in positive even dimensions). Thus, the commutative diagram 

\begin{picture}(0,50)(-25,-5)
\put(8,32){$H^*(\P^\infty)\otimes H^*(\P^\infty;\mathbb{Z}/2)$}
\put(-28,0){$H^*(\P^\infty)\otimes H^*(\P^\infty;H^3(\homeo)$}
\put(45,27){\vector(0,-1){15}}
\put(148,0){$H^*(\P^\infty;H^3(\homeo)$}
\put(175,32){$H^*(\P^\infty;\mathbb{Z}/2)$}
\put(198,27){\vector(0,-1){15}}
\put(107.5,34){\vector(1,0){52}}
\put(125,2){\vector(1,0){16}}
\end{picture}

\noindent of $H^*(\P^\infty)$-actions implies that multiplication by 
$\omega\in H^2(\P^\infty)$ is monic on the 2-torsion part of 
$H^*(\P^\infty;H^3(\homeo)$.
}\end{nota}

Here is a chart of the $E_2$-term in~(\ref{CLSSS}):

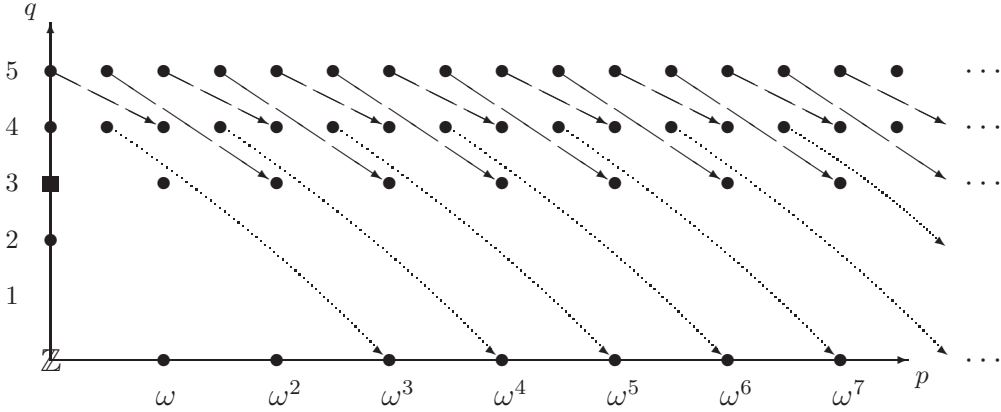
\begin{picture}(0,118)(-18,-17)
\put(2,0){\vector(1,0){228}}
\put(2,0){\vector(0,1){90}}
\put(0,30){$\bullet$}
\put(0,45){\rule{2mm}{2mm}}
\multiput(30,45)(30,0){7}{$\bullet$}
\multiput(0,60)(15,0){16}{$\bullet$}
\multiput(0,75)(15,0){16}{$\bullet$}
\multiput(30,-2)(30,0){7}{$\bullet$}
\put(-10,15){\footnotesize$1$}
\put(-10,30){\footnotesize$2$}
\put(-10,45){\footnotesize$3$}
\put(-10,60){\footnotesize$4$}
\put(-10,75){\footnotesize$5$}
\put(-5,92){\footnotesize$q$}
\put(232,-6){\footnotesize$p$}
\put(245,45){$\cdots$}
\put(245,60){$\cdots$}
\put(245,75){$\cdots$}
\put(245,-2){$\cdots$}
\put(30,-12){$\omega$}
\put(60,-12){$\omega^2$}
\put(90,-12){$\omega^3$}
\put(120,-12){$\omega^4$}
\put(150,-12){$\omega^5$}
\put(180,-12){$\omega^6$}
\put(210,-12){$\omega^7$}
\put(-.7,-2.8){$\mathbb{Z}$}
\multiput(2,77)(10,-5){2}{\line(2,-1){8}}
\put(22,67){\vector(2,-1){8}}
\multiput(32,77)(10,-5){2}{\line(2,-1){8}}
\put(52,67){\vector(2,-1){8}}
\multiput(62,77)(10,-5){2}{\line(2,-1){8}}
\put(82,67){\vector(2,-1){8}}
\multiput(92,77)(10,-5){2}{\line(2,-1){8}}
\put(112,67){\vector(2,-1){8}}
\multiput(122,77)(10,-5){2}{\line(2,-1){8}}
\put(142,67){\vector(2,-1){8}}
\multiput(152,77)(10,-5){2}{\line(2,-1){8}}
\put(172,67){\vector(2,-1){8}}
\multiput(182,77)(10,-5){2}{\line(2,-1){8}}
\put(202,67){\vector(2,-1){8}}
\multiput(212,77)(10,-5){2}{\line(2,-1){8}}
\put(232,67){\vector(2,-1){8}}
\multiput(17,77)(15,-10){2}{\line(3,-2){13}}
\put(47,57){\vector(3,-2){13}}
\multiput(47,77)(15,-10){2}{\line(3,-2){13}}
\put(77,57){\vector(3,-2){13}}
\multiput(77,77)(15,-10){2}{\line(3,-2){13}}
\put(107,57){\vector(3,-2){13}}
\multiput(107,77)(15,-10){2}{\line(3,-2){13}}
\put(137,57){\vector(3,-2){13}}
\multiput(137,77)(15,-10){2}{\line(3,-2){13}}
\put(167,57){\vector(3,-2){13}}
\multiput(167,77)(15,-10){2}{\line(3,-2){13}}
\put(197,57){\vector(3,-2){13}}
\multiput(197,77)(15,-10){2}{\line(3,-2){13}}
\put(227,57){\vector(3,-2){13}}
\put(0,0){\begin{picture}(0,0)
\qbezier[60](18,62)(60,33)(89,3)
\put(90.6,1.6){\vector(4,-3){0}}
\end{picture}}
\put(30,0){\begin{picture}(0,0)
\qbezier[60](18,62)(60,33)(89,3)
\put(90.6,1.6){\vector(4,-3){0}}
\end{picture}}
\put(60,0){\begin{picture}(0,0)
\qbezier[60](18,62)(60,33)(89,3)
\put(90.6,1.6){\vector(4,-3){0}}
\end{picture}}
\put(90,0){\begin{picture}(0,0)
\qbezier[60](18,62)(60,33)(89,3)
\put(90.6,1.6){\vector(4,-3){0}}
\end{picture}}
\put(120,0){\begin{picture}(0,0)
\qbezier[60](18,62)(60,33)(89,3)
\put(90.6,1.6){\vector(4,-3){0}}
\end{picture}}
\put(150,0){\begin{picture}(0,0)
\qbezier[60](18,62)(60,33)(89,3)
\put(90.6,1.6){\vector(4,-3){0}}
\end{picture}}
\put(180,0){\begin{picture}(0,0)
\qbezier[40](18,62)(41,48)(57.5,32.5)
\put(60,30.5){\vector(4,-3){0}}
\end{picture}}
\end{picture}

\noindent A bullet represents a copy of $\mathbb{Z}/2$, and the solid square
in $E_2^{0,3}$ stands for $\mathbb{Z}\oplus\mathbb{Z}/2$. Note the group 
$\mathbb{Z}$ at node $(0,0)$. The dashed and dotted arrows are to be ignored 
for now---they will be shown to give the pattern of differentials 
in the spectral sequence.

\medskip
One immediately sees from the chart that the first group in the short exact 
sequence (analogous to~(\ref{portion}))
\begin{equation}\label{luno}
H^1(\bp)\to H^1(\bp;\mathbb{Z}/2)\to H^2(\bp)\stackrel{2}{\to} H^2(\bp)
\end{equation}
is trivial. Coupled with the fact---coming from Table~\ref{T2}---that 
the second group in~(\ref{luno}) is $\mathbb{Z}/2\oplus\mathbb{Z}/2$, this
implies that the element at node $(0,2)$ is a permanent cycle (this much
is obvious since there is no possible nontrivial target for it), and that
$H^2(\bp)=\mathbb{Z}/2\oplus\mathbb{Z}/2$, a trivial extension 
in this part of the $E_\infty$-term. 

\medskip
Next observe that the only potentially nontrivial 
differentials $d_r$ originating at node $(0,3)$
land at node $(4,0)$---$\,d_4$-differentials possibly hitting $\omega^2$. 
The kernel of such a differential is $H^3(B(\P^3,2))$,
whose only possibilities are:
\begin{itemize}
\item[(i)] $\mathbb{Z}\oplus\mathbb{Z}/2$, 
if  the torsion element at node $(0,3)$ 
is a $d_4$-cycle (independently of the $d_4$-image of the torsion-free 
element at node $(0,3)$);
\item[(ii)] $\mathbb{Z}$, if the torsion element 
at node $(0,3)$ hits, under $d_4$, 
the element at node $(4,0)$ (once again, 
independently of the $d_4$-image of the torsion-free 
element at node $(0,3)$).
\end{itemize}
Either one of these two possible cases gives at most one nontrivial
$\Sq^1$-cohomology class in $H^3(B(\P^3,2);\mathbb{Z}/2)$---the one coming
from the integral class. But the calculation at the end of 
Examples~\ref{ejemp2} shows that the $\Sq^1$-cohomology has rank 2 in dimension
$3$, so that $H^4(B(\P^3,2))$ must have an element of order $4$. Therefore
both elements in total degree 4 in the chart must survive to nontrivial classes
in $E_\infty$ making up a nontrivial extension in $H^4(B(\P^3,2))=\mathbb{Z}/4$.
A number of consequences follow at this point:
\begin{itemize}
\item[(a)] All elements at nodes $(0,3)$ and $(0,4)$ are permanent cycles.
\item[(b)] Possibility (i) above holds, that is, $H^3(B(\P^3,2))=\mathbb{Z}
\oplus\mathbb{Z}/2$.
\item[(c)] $\omega^2$ is the nontrivial element of order 2 in 
$H^4(\bp)=\mathbb{Z}/4$---proving Theorem~\ref{notrivialidad} and the 
second part in Theorem~\ref{lacohofinal}.
\end{itemize}

The only task remaining in the proof of Theorem~\ref{lacohofinal} 
is the determination of $H^5(B(\P^3,2))$. This can be obtained in terms of the 
non-orientability (proved in~\cite{D8}) of the $5$-manifold mentioned  
in~(\ref{finito}) to have the homotopy type of $B(\P^3,2)\,$---or, 
alternatively, by using~\cite[Theorem~43]{bausum}.
Instead, having come this far, we finish up the description of the spectral 
sequence~(\ref{CLSSS}).

\medskip
Together with the $H^*(\P^\infty)$-module structure described in
Remark~\ref{nmodule}, (a) above implies that 
all elements in the ($q=3$)-line, as well as all elements at nodes
$(2i,4)$ for $i\geq0$ are permanent cycles. But the element
at node $(2,4)$ cannot survive to a nontrivial class in $E_\infty$
(in view of~(\ref{finito})), and this forces a nontrivial $d_2$-differential
from node $(0,5)$ to node $(2,4)$. In fact, the $H^*(\P^\infty)$-module 
structure implies that all differentials 
$d_2\colon E_2^{2i,5}\to E_2^{2i+2,4}$ are isomorphisms for 
$i\geq0\,$---the family of $d_2$-differentials depicted in the chart above.
Furthermore, the only other potentially nontrivial $d_2$-differentials are
those of the form 
\begin{equation}\label{lasno}
d_2^{2i+1,5}\colon E_2^{2i+1,5}\to E_2^{2i+3,4} 
\end{equation}
for $i\geq0$. But if $d_2^{2i+1,5}$ 
was nontrivial, then there would not be any class left
to kill the permanent cycle at node $(2i+4,3)$---which must 
be killed in view of~(\ref{finito}). Therefore, all differentials 
in~(\ref{lasno}) are actually trivial and, instead, all differentials
$d_3^{2i+1,5}\colon E_2^{2i+1,5}\to E_2^{2i+4,3} $ ($i\geq0$) are
isomorphisms, giving the family of $d_3$-differentials depicted in the chart 
above (these, by the way, are compatible with the $H^*(\P^\infty)$-module
structure). This accounts for all the possible nontrivial differentials
up to $E_5$ (recall that every element in the ($q=3$)-line is a 
permanent cycle). Finally, the family of $d_5$-differentials 
depicted in the chart above is forced since classes $\omega^i$ with $i\geq3$ 
must be killed in view of~(\ref{finito}) (once again, the resulting family of
$d_5$-differentials is compatible with the $H^*(\P^\infty)$-module structure). 
In particular, in total degrees 5 there is only one remaining $\mathbb{Z}/2$, 
which makes up $H^5(B(\P^3,2))$.

\begin{nota}\label{estrult}{\em
Some of the properties in the multiplicative structure of the ring 
$H^*(B(\P^3,2))$ can be recovered from the analysis above.
To begin with, we know that $\omega\in H^2(B(\P^3,2))$
has $\omega^2\neq0$ but $\omega^3=0$. Next, it is asserted 
that there is an nontrivial exterior element
$a\in H^2(B(\P^3,2))$ (so that $a$ and $\omega$
form a $\mathbb{Z}/2$-basis of $H^2(B(\P^3,2))$).
Indeed, pick any $a\not\in\{0,\omega\}$; 
if $a^2\neq0$, then $a^2=\omega^2$ would be forced, so that
$(a+\omega)^2=0\,$. Unfortunately, once an exterior $a$ has been fixed,
the author does not know how to decide
whether $\omega a=0$ or $\omega a=\omega^2$
(both possibilities are compatible
with the multiplicative structure seen from the spectral sequence).
Further, there are the three indecomposable elements
\begin{itemize}
\item $b\in H^3(B(\P^3,2))$, generating a $\mathbb{Z}$;
\item $c\in H^3(B(\P^3,2))$, generating a $\mathbb{Z}/2$;
\item $d\in H^4(B(\P^3,2))$, generating a $\mathbb{Z}/4$.
\end{itemize}
The relation $\omega^2=2d$ has already been discussed, but now the fact
that $\omega c$ is the generator of $H^5(B(\P^3,2))=\mathbb{Z}/2$
follows from the multiplicative structure in the 
spectral sequence. As an exercise, the interested reader can check that
the explicit formulas in~\cite[page~25]{coho} imply that, in
the $H^*(\P^\infty)$-module structure discussed in Remark~\ref{nmodule},
multiplication by $\omega\in H^*(\P^\infty)$ sends the torsion-free 
generator in $H^0(\P^\infty;H^3(\homeo))$ into the generator 
in $H^2(\P^\infty;H^3(\homeo))$, so that $\omega b=\omega c$.
The author does not know how to deal with the 
two remaining products $a b$ and $a c$ (as in the case of $\omega a$, 
their triviality in the spectral sequence just indicates
that these products have filtration higher than expected).
}\end{nota}

\section{Symmetric motion planners}\label{smp}
M.~Farber began in~\cite{F1,farberinstabilities} a study of the 
continuity instabilities inherent in any motion planner for a robotical system.
In this section, 
his methods and results are (partially) adapted to the symmetric case.

\smallskip
Suppose a given robotical system has to be programmed to
perform tasks which, however, need to be decided during the 
course of the operation. Thus, the programming must be made in such a way that, 
after being fed with a given pair of states $(A,B)$ of the system, 
the robot  
should decide and perform, in an autonomous way, the required transformations 
for going from one of the given states to the other. There are three natural 
requirements that arise in many practical situations: 
\begin{enumerate}
\item \label{symetricmotion} 
Motion should be symmetric: the chosen
movement from $A$ to $B$ should be
the same one, but in reverse direction, as the movement
from $B$ to $A$. 
\item \label{irrelevantmotion}
There is no need to plan motion from a given state to itself:
the robot will only be fed with pairs of
different states ($A\neq B$). 
\item \label{robustness}
The programming should be robust enough to allow for small ``errors'' 
in the description of the states:
if the states $A$ and $B$ change into slightly perturbed new states 
$A'$ and $B'$, then the corresponding movements $A\leadsto B$ and $A'
\leadsto B'$ should be ``roughly'' the same. 
\end{enumerate}

Let $X$ be the space of all possible states the system can take. 
The topology of $X$ is determined by the capabilities and constraints
of the system (e.g.~physical design of the robot, or obstacles
in the robot's path).
Attention will be restricted to the case of a path-connected $X$,
which means that the system can always be transformed from any 
given state to any other state. Then, the required 
programming is encoded by a continuous (so that condition~\ref{robustness} 
above holds) $\mathbb{Z}/2$-equivariant (so that condition~\ref{symetricmotion} 
above holds) section $s\colon F(X,2)\to P_1(X)$ of the evaluation map 
$\ev_1\colon P_1(X)\to F(X,2)$, where the diagonal has been removed
from $X\times X$ to form $F(X,2)$ in view of condition~\ref{irrelevantmotion} 
above (see Section~\ref{sectctcsgoodwillie} for the relevant definitions about 
$\ev_1$). However, it is known that such a 
programming problem is solvable only in very special situations; indeed,
the required section $s$ can exist only when $X$ has a two-sided unital 
homotopy comultiplication\footnote{Such a property on
$X$ follows from~\cite{whitehead} and the inequalities $\cat\leq\TC\leq\TC^S$. 
In fact, the considerations after Example 7 in~\cite{FGsymm} mention 
the appealing possibility that such a section $s$ exists if and only if $X$ is 
contractible.} (e.g.~$X$ a suspension). Thus, the best one can hope 
for, in general, is to be able to give a small number of continuous
local moving instructions (with a low {\it order of instability}, see 
the considerations previous to Proposition~\ref{lema91} below),
that is, being able to partition the space $F(X,2)$ into
a small number of pieces $F_1,\ldots,F_k$ (with non-empty multiple 
intersection of closures $\overline{F}_{i_1}\cap\cdots\cap \overline{F}_{i_t}$ 
only for small $t$), each admitting a continuous $\mathbb{Z}/2$-equivariant 
local section $s_i\colon F_i\to P_1(X)$ for $\ev_1$.

\begin{definicion}\label{smpdef}
A symmetric motion planner for $X$, ${\cal M}=\{F_i,s_i\}$, 
consists of a collection $F_1,\ldots,F_k$ (called the local domains of 
${\cal M}$) of subsets of $F(X,2)$ and a collection of maps
$s_i\colon F_i\to P_1(X)$ (called the local rules of ${\cal M}$). 
This data is required to satisfy:
\begin{itemize}
\item[{\em(a)}] the local domains form a partition of $F(X,2)$ (the $F_i$'s 
cover $F(X,2)$ and are pairwise disjoint);
\item[{\em(b)}] each local domain should be a $\mathbb{Z}/2$-equivariant 
neighborhood deformation retract of $F(X,2)$;
\item[{\em(c)}] each local rule $s_i$ should be a continuous 
$\mathbb{Z}/2$-equivariant section of ${\ev_1}_{|F_i}$.
\end{itemize}
\end{definicion}

Let us spell out part (b) in Definition~\ref{smpdef}. 
First, each $F_i$ must be stable 
under the involution in $F(X,2)$. In addition, $F_i$ must have an 
open neighborhood $U_i$ in $F(X, 2)$ which is required to be stable under 
the involution and to admit a $\mathbb{Z}/2$-equivariant retraction 
$r_i\colon U_i\to F_i$. We also require that as a map to $F(X, 2)$, $r_i$ 
should be
$\mathbb{Z}/2$-equivariantly deformable to the inclusion $U_i
\hookrightarrow F(X,2)$. The last (rather technical, but natural in 
practical settings) condition captures the essence 
in~\cite{farberinstabilities}---where the non-symmetric 
situation is controlled by requiring the 
less restrictive hypothesis that each local domain be an ENR.
The present requirement is made not only in order to 
avoid pathological situations, but to have good control on the topology.

\smallskip
A symmetric motion planner gives a
practical algorithm to approach the 
programming problem posed at the beginning 
of the section: just determine the local domain $F_i$ containing the given 
pair of states $(A,B)$, and apply the local rule
$s_i(A,B)$ in the desired direction---note
that $(B,A)$ lies in $F_i$ too, and that
$s_i(B,A)$ runs the path $s_i(A,B)$ in reverse direction, 
i.e.~$s_i(A,B)(t)=s_i(B,A)(1-t)$. It is to be observed, however, that
such an algorithm will not necessarily 
satisfy the requirement \ref{robustness} above:
the closure of two (or more) local domains might have a non-empty 
intersection---a problem inherent to the motion planning task.
In~\cite{farberinstabilities}, this situation led Farber 
to the concept of order of instability of a motion planner,
and to its connection (via $\TC$) with the number of
local rules in (non-necessarily symmetric) motion planners.
The author hopes to deal elsewhere with a possible symmetric analogue 
of this phenomenon. As a preliminary step, 
it is next shown that $\TC^S(X)$ is a sharp lower bound for the 
number of local rules of symmetric motion planners for $X$.

\smallskip
Theorem~\ref{tcsvsslr} is a direct consequence of 
Propositions~\ref{lema91} and~\ref{elotro} below which, in turn,
are proved with arguments inspired by the proof of part~(1) of Theorem~6.1 
in~\cite{farberinstabilities}.

\begin{proposicion}\label{lema91}
The number of local domains in any symmetric motion planner of 
$X$ is bounded from below by $\TC^S(X)$.
\end{proposicion}
\begin{proof}
Let $F\subset F(X,2)$ be a local domain of a given symmetric motion planner
${\cal M}$ with $k$ local rules, and 
let $s\colon F\to P_1(X)$ be the local rule corresponding to $F$. 
Choose an open neighborhood $U$ of $F$ which is stable 
under the involution in $F(X,2)$, and admits
a retraction $r\colon U\to F$ which is $\mathbb{Z}/2$-equivariantly
deformable to the inclusion $U\hookrightarrow F(X,2)$.
Under these conditions the image $U'$ of $U$ under the projection 
$F(X,2)\to B(X,2)$ is open and, as shown below, there is 
a continuous $\mathbb{Z}/2$-equivariant section 
$\sigma\colon U\to P_1(X)$ for $\ev_1$. Granting this and passing to 
orbit spaces, $\sigma$ determines a continuous section $\sigma'$ for $\ev_2$ 
on $U'$. When this construction is performed over each of the 
local domains of ${\cal M}$, there results an open cover $U'_1,\ldots,
U'_k$ of $B(X,2)$, where each $U'_i$ admits a local section $\sigma'_i$ 
for $\ev_2$. Thus $k\geq\TC^S(X)$, as asserted.

\smallskip
In order to define the required section $\sigma$ at a pair 
$(x,y)\in U$, consider 
the path $H_{x,y}\colon[0,1]\to F(X,2)$ given by $H_{x,y}(t)=H(x,y,t)$,
where $H\colon U\times [0,1]\to F(X,2)$ is a fixed $\mathbb{Z}/2$-equivariant 
homotopy between the inclusion $U\hookrightarrow F(X,2)$ 
(at $t=0$) and the retraction $r\colon U\to F$ (at $t=1$). Under 
these conditions set $\sigma(x,y)$ to be the concatenation 
$$\sigma(x,y)=\gamma_{x,y}\cdot s(\gamma_{x,y}(1),\delta_{x,y}(1))\cdot
\bar{\delta}_{x,y}\,,$$
where $\bar{\delta}_{x,y}$ is the path $\delta_{x,y}$ in reverse, and 
$\gamma_{x,y},\delta_{x,y}\colon[0,1]\to F(X,2)$ are the components
of $H_{x,y}=(\gamma_{x,y},\delta_{x,y})$. The 
hypothesis on $H$ means that $\gamma_{x,y}=\delta_{y,x}$ and $\delta_{x,y}=
\gamma_{y,x}$, so that $\sigma$ is equivariant as required.
\end{proof}

\begin{proposicion}\label{elotro}
If $X$ is a smooth manifold, then there is a symmetric motion planner of 
$X$ with $\TC^S(X)$ local rules.
\end{proposicion}
\begin{proof}
Let $V_1,\ldots, V_k$ be an open cover of $B(X,2)$ with local sections 
$\sigma_i\colon V_i\to P_2(X)$ of $\ev_2$, where $k=\TC^S(X,2)$. The 
assumption $k>1$ can safely be made in view 
of~\cite[Lemma~8]{FGsymm}. Choose a smooth partition of unity 
\begin{equation}\label{partition}
f_i\colon B(X,2)\to[0,1], \quad i=1,\ldots,k,
\end{equation}
subordinate to the cover. Using Sard's Theorem it is possible to 
show\footnote{The author thanks 
Peter Landweber for pointing out the need of the condition $k>1$, 
and for explaining the details of this assertion.} the existence of
numbers $c_1,\ldots,c_k\in(0,1)$ satisfying
$\sum c_i=1$ and such that each $c_i$ is a regular value of $f_i$. 
For $i=1,\ldots,k$, consider the subspaces $G_i$ of $B(X,2)$ consisting of sets
$\{x,y\}\in B(X,2)$ satisfying $f_i(\{x,y\})\geq c_i$ and $f_j(\{x,y\})<c_j$ 
for $j<i$. One checks that each $G_i$ is a submanifold (with possibly 
non-empty boundary) of the corresponding $V_i$, and that the collection
$\{G_1,\ldots,G_k\}$ is a partition of $B(X,2)$. Pick a tubular (normal)
neighborhood $N_i$ of $G_i$ in $V_i$ (so that $N_i$ contains $G_i$ as a 
deformation retract). Let $\pi\colon F(X,2)\to B(X,2)$ be the canonical 
projection and, for $i=1,\ldots, k$, consider the inverse images 
$F_i=\pi^{-1}(G_i)$, a partition of $F(X,2)$, and the corresponding open 
neighborhoods $U_i=\pi^{-1}(N_i)$, all of which are stable under the 
involution in $F(X,2)$. As explained in~\cite[Lemma~8]{FGsymm}, each 
restricted section ${\sigma_i}_{|N_i}$ determines a continuous 
$\mathbb{Z}/2$-equivariant section $s_i\colon U_i\to P_1(X)$ of $\ev_1$. 
Therefore, to conclude the proof, it remains to find $\mathbb{Z}/2$-equivariant 
retractions $r_i\colon U_i\to F_i$ which are $\mathbb{Z}/2$-equivariantly 
deformable to the corresponding inclusions $U_i\hookrightarrow F(X,2)$. 
This is done in the next paragraph
by a standard homotopy lifting argument for the covering projection
$\pi\colon F(X,2)\to B(X,2)$.

For each $i=1,\ldots,k$, 
choose a homotopy $H_i\colon N_i\times [0,1]\to B(X,2)$ deforming the 
inclusion $H_i(-,0)\colon N_i\hookrightarrow B(X,2)$ 
to a retraction $H_i(-,1)\colon N_i\to G_i$. The homotopy lifting property of
$\pi$ applied to the commutative diagram

\begin{picture}(0,70)(-58,-3)
\put(0,10){$U_i\times [0,1]$}
\put(2,50){$U_i\times \{0\}$}
\put(66,10){$N_i\times [0,1]$}
\put(130,10){$B(X,2)$}
\put(130,50){$F(X,2)$}
\put(38,12){\vector(1,0){24}}
\put(105,12){\vector(1,0){22}}
\put(37,52){\vector(1,0){89}}
\put(143,44){\vector(0,-1){23}}
\put(14,44){\vector(0,-1){23}}
\put(37,53.5){\oval(3,3)[l]}
\put(12.5,44){\oval(3,3)[t]}
\put(38,4){\scriptsize $\pi\times [0,1]$}
\put(112,4){\scriptsize $H_i$}
\qbezier[40](25,22)(75,33)(124,45)
\put(128,46.3){\vector(4,1){0}}
\put(62,35){\scriptsize $\widehat{H}_i$}
\put(146,32){\scriptsize $\pi$}
\end{picture}

\noindent
yields the dotted homotopy $\widehat{H}_i\colon U_i\times [0,1]\to F(X,2)$
compatible with $H_i$ under $\pi$, and
deforming the inclusion $\widehat{H}_i(-,0)\colon U_i\hookrightarrow F(X,2)$
to the required retraction $\widehat{H}_i(-,1)\colon U_i\hookrightarrow F_i$.
To see that each branch of $\widehat{H}_i$ is $\mathbb{Z}/2$-equivariant,
fix a point $(a,b)\in U_i$ and observe that, if $\,\widehat{H}_i(a,b,t)=(x,y)$
for some $t\in[0,1]$, then $\widehat{H}_i(b,a,t)\in\{(x,y),(y,x)\}$ is forced.
But we need to see that $\widehat{H}_i(b,a,t)=(y,x)$ must in fact be 
the case. To this end, note that
$\widehat{H}_i(a,b,-)$ and $\widehat{H}_i(b,a,-)$ are two paths $[0,1]\to
F(X,2)$, starting respectively at $(a,b)$ and $(b,a)$, and projecting 
under $\pi$ to the same path. So, by the unique path lifting property of 
$\pi$, $\widehat{H}_i(a,b,-)$ and $\widehat{H}_i(b,a,-)$ are point-wise 
different.
\end{proof}

The paper closes by exploiting the main idea in~\cite{symmotion} in order to 
sketch a $5$-local-rules
symmetric motion planner for any autonomous robot whose state space 
is $\P^3=\mathrm{SO}(3)$. 

\smallskip
Start by consider the standard Euclidean 
charts $$U_i=\{\left[x_0,\ldots, x_4\right]
\in\P^4\,|\,\sum_j x_j^2=1\,\,\mathrm{and}\,\,
x_i\neq0\},\quad 0\leq i\leq4,$$ of $\P^4$, and the  
corresponding local sections $s_i\colon U_i\to S^4$ of the 
canonical projection $q\colon S^4\to\P^4$, where 
\begin{equation}\label{secciones}
s_i\left(\left[x_0,\ldots,x_4\right]\right)=\frac{|x_i|}{x_i}
\left(x_0,\ldots,x_4\right).
\end{equation}
According to~\cite{VG}, the formula
\begin{equation}\label{efe}
f\left(\left[z_0,z_1\right]
\right)=\frac{\left(z_0^2,z_1^2,\Re(z_0z_1)
\right)}{\sqrt{1-|z_0z_1|^2}-\Im(z_0z_1)}
\end{equation}
determines an explicit embedding $f\colon\P^3\subset\mathbb{R}^5$,
where $\Re(\omega)$ and $\Im(\omega)$ stand, respectively, for the
real and imaginary parts of a complex number $\omega$, and where
an element in $S^3$ has been represented by a pair of complex numbers $z_i$ 
($i=0,1$) with $|z_0|^2+|z_1|^2=1$. Using
Haefliger's formula~(\ref{haefligermap}),~(\ref{efe}) gives an explicit 
$\mathbb{Z}/2$-equivariant map $H\colon F(\P^3,2)\to S^4$
determining the pull-back diagram

\begin{picture}(0,60)(-101,8)
\put(0,17){$B(\P^3,2)$}
\put(0,50){$F(\P^3,2)$}
\put(63,50){$S^4$}
\put(63,17){$\P^4$}
\put(33,20){\vector(1,0){26}}
\put(33,53){\vector(1,0){26}}
\put(15,45){\vector(0,-1){18}}
\put(66,45){\vector(0,-1){18}}
\put(42,22.5){\scriptsize $\widetilde{H}$}
\put(42,55.5){\scriptsize $H$}
\put(8,35){\scriptsize $\pi$}
\put(69,35){\scriptsize $q$}
\end{picture}

\noindent where $\pi$ stands for the canonical projection. In these terms,
each section~(\ref{secciones}) pulls back to an explicit local section 
$\sigma_i\colon\widetilde{H}^{-1}(U_i)\to F(\P^3,2)$ for $\pi$. For instance,
$\widetilde{H}^{-1}(U_0)$ consists of those $\left\{\rule{0mm}{3.5mm}
[z_0,z_1],[\omega_0,
\omega_1]\right\}\in B(\P^3,2)$ such that
\begin{equation}\label{condicion}
\left(\sqrt{1-|\omega_0\omega_1|^2}-\Im(\omega_0\omega_1)\right)\Re(z_0^2)-
\left(\sqrt{1-|z_0z_1|^2}-\Im(z_0z_1)\right)\Re(\omega_0^2)\neq0
\end{equation}
whereas 
$$
\sigma_0\left(\left\{\rule{0mm}{3.5mm}[z_0,z_1],[\omega_0,\omega_1]\right\}
\right)=\begin{cases} \left(\rule{0mm}{3.5mm}[z_0,z_1],
[\omega_0,\omega_1]\right), & \mbox{left-hand-side 
of~(\ref{condicion}) is positive;}\\ \left(\rule{0mm}{3.5mm}
[\omega_0,\omega_1],[z_0,z_1]\right), & \mbox{left-hand-side 
of~(\ref{condicion}) is negative.}\end{cases}
$$

\noindent
But according to~\cite[Propositions~2.1 and~2.2]{symmotion} there is a 
well-defined commutative diagram

\begin{picture}(0,54)(-56,11)
\put(70,17){$B(\P^3,2)$}
\put(-2,50){$F(\P^3,2)$}
\put(141,50){$P_2(\P^3)$}
\put(52,50){$\frac{\left\{(x,y)\in S^3\times S^3\,|\,x\neq\pm y\right\}}{(x,y)
\sim(y,x)\sim(-x,-y)}$}
\put(32,52.5){\vector(1,0){15}}
\put(36.5,55){\scriptsize $\Psi$}
\put(122,52.5){\vector(1,0){15}}
\put(128,56){\scriptsize $g$}
\put(30,43.5){\vector(2,-1){36}}
\put(42,29){\scriptsize $\pi$}
\put(140,43.5){\vector(-2,-1){36}}
\put(122.5,29){\scriptsize $\ev_2$}
\end{picture}

\noindent where horizontal maps are explicitly given by
$$\Psi\left([x],[y]\right)=\left[\,\frac{x+y}{||x+y||}\,,\,\frac{x-y}{||x-y||}
\,\right]$$ and $$g([x,y])(t)=\left[\frac{ty+(1-t)x}{||ty+(1-t)x||}\right]$$
for $x,y\in S^3$.
(Note that a typo occurs in~\cite{symmotion} four lines above formula~(5),
as well as 7 lines below formula~(13):
the inequality sign in the definition of $\widetilde{\Delta}$
should be replaced by an equality sign.) Therefore, on the open cover 
\begin{equation}\label{cover}
\{\widetilde{H}^{-1}(U_0),\ldots,\widetilde{H}^{-1}(U_4)\},
\end{equation} 
the composites
$g\circ\Psi\circ\sigma_i$ for $i=0,\ldots,4$ give $5$ explicit local sections
for $\ev_2$, and the required symmetric motion planner is then described
by the proof of Proposition~\ref{elotro}. 

\smallskip
The above process has only
one non-constructive component, namely, no suitable
smooth partition of unity~(\ref{partition}) is explicitly given. But a closer
look at the proof of Proposition~\ref{elotro} shows that the critical goal
is to refine~(\ref{cover}) to a pair-wise disjoint cover by submanifolds 
of $B(\P^3,2)$. The author encourages readers interested in implementing 
this symmetric motion planner to attempt to provide the missing explicit 
construction. At any rate, the actual local rules have been 
explicitly described: after pulling back over $F(\P^3,2)\to B(\P^3,2)$, 
these are given by restrictions of the composites $g\circ\Psi\circ\sigma_i$.

\bigskip
Jes\'us Gonz\'alez\quad {\tt jesus@math.cinvestav.mx}

{\sl Departamento de Matem\'aticas, CINVESTAV--IPN

Apartado Postal 14-740 M\'exico City, C.P. 07000, M\'exico}

\begin{thebibliography}{10}

\bibitem{Astey}
L.~Astey,
``Geometric dimension of bundles over real projective spaces'',
{\em Quart. J. Math. Oxford Ser. (2)} {\bf 31} (1980) 139--155.

\bibitem{asteyobstr}
L.~Astey,
``A cobordism obstruction to embedding manifolds'',
{\em Illinois J. Math.} {\bf 31} (1987) 344--350.

\bibitem{bausum}
D.~R.~Bausum,
``Embeddings and immersions of manifolds in Euclidean space'',
{\em Trans. Amer. Math. Soc.} {\bf 213} (1975) 263--303. 

\bibitem{crabb}
M.~C.~Crabb,
{\em $Z/2$-homotopy theory},
London Mathematical Society Lecture Note Series 44. 
Cambridge University Press, Cambridge-New York, 1980.

\bibitem{Davistrong}
D.~M.~Davis,
``A strong nonimmersion theorem for real projective spaces'',
{\em Ann. of Math. (2)} {\bf 120} (1984) 517--528.

\bibitem{dontables}
D.~M.~Davis,
``Table of immersions and embeddings of real projective spaces'',
available from {\tt http://www.lehigh.edu/\mbox{$\sim$}dmd1/immtable}

%\bibitem{dold}
%A.~Dold, {\it Lectures on Algebraic Topology}
%Second edition. Grundlehren der Mathematischen Wissenschaften 200. 
%Springer-Verlag, Berlin-New York, 1980.

\bibitem{coho}
L.~Evens,
{\em The Cohomology of Groups},
Oxford Mathematical Monographs. Oxford Science Publications. 
The Clarendon Press, Oxford University Press, New York, 1991.

\bibitem{F1}
M.~Farber,
``Topological complexity of motion planning'',
{\em Discrete Comput. Geom.} {\bf 29} (2003) 211-221.

\bibitem{farberinstabilities}
M.~Farber,
``Instabilities of robot motion'',
{\em  Topology Appl.} {\bf 140} (2004) 245--266.

\bibitem{FGsymm}
M.~Farber and M.~Grant,
``Symmetric motion planning'',
in Topology and Robotics, {\em Contemp. Math.} {\bf 438}, 
Amer. Math. Soc., Providence, RI (2007) 85--104.

\bibitem{FTY}
M.~Farber, S.~Tabachnikov, and S.~Yuzvinsky,
``Topological robotics: motion planning in projective spaces'',
{\em  Int. Math. Res. Not.} {\bf 34} (2003) 1853--1870. 

\bibitem{feder}
S.~Feder,
``The reduced symmetric product of projective spaces and the 
generalized Whitney theorem'',
{\em  Illinois J. Math.} {\bf 16} (1972) 323--329. 

\bibitem{FT}
Y.~F\'elix and D.~Tanr\'e,
``The cohomology algebra of unordered configuration spaces'',
{\em J. London Math. Soc. (2)} {\bf 72} (2005) 525--544.

\bibitem{sam}
S.~Gitler and D.~Handel,
``The projective Stiefel manifolds I'',
{\em Topology} {\bf 7} (1968) 39--46. 

\bibitem{symmotion}
J.~Gonz\'alez and P.~Landweber,
``Symmetric topological complexity of projective and lens spaces'',
{\em Algebr. Geom. Topol.} {\bf 9} (2009) 473--494.

\bibitem{D8}
J.~Gonz\'alez and P.~Landweber,
``On the integral cohomology ring of configuration spaces
of pairs of points in real projective spaces'',
under preparation.

\bibitem{GKW}
T.~G.~Goodwillie, J.~R.~Klein, and M.~S.~Weiss,
``Spaces of smooth embeddings, disjunction and surgery'',
in Surveys on Surgery Theory, {\em Ann. of Math. Stud.} {\bf 149}
(2001) 221--284.

\bibitem{GKWmodels}
T.~G.~Goodwillie, J.~R.~Klein, and M.~S.~Weiss,
``A Haefliger style description of the embedding calculus tower'',
{\em Topology} {\bf 42} (2003) 509--524.

\bibitem{GW}
T.~G.~Goodwillie and M.~S.~Weiss,
``Embeddings from the point of view of immersion theory, Part II'',
{\em Geom. Topol.} {\bf 3} (1999) 103--118.

\bibitem{haefliger}
A.~Haefliger,
``Points multiples d'une application et produit cyclique r\'eduit'',
{\em  Amer. J. Math.} {\bf 83} (1961) 57-70.

\bibitem{handel}
D.~Handel,
``An embedding theorem for real projective spaces'',
{\em  Topology} {\bf 7} (1968) 125--130. 

\bibitem{hatcher}
A.~Hatcher, 
{\em Algebraic Topology}, 
Cambridge University Press, Cambridge, 2002.

\bibitem{MR0440554}
J.~W.~Milnor and J.~D.~Stasheff,
{\it Characteristic Classes}
Annals of Mathematics Studies No. 76, Princeton University Press,
Princeton N.J., 1974.

\bibitem{munson}
B.~A.~Munson,
``Embeddings in the $3/4$ range'',
{\em  Topology} {\bf 44} (2005) 1133--1157.

\bibitem{rees}
E.~Rees,
``Embeddings of real projective spaces'', 
{\em  Topology} {\bf 10} (1971) 309--312.

\bibitem{schwarz}
A.~S.~Schwarz,
``The genus of a fiber space'',
{\em Amer. Math. Soc. Transl. Ser. 2} {\bf 55} (1966) 49--140.

\bibitem{stolz}
S.~Stolz,
``The level of real projective spaces'',
{\em Comment. Math. Helv.} {\bf 64} (1989) 661--674.

\bibitem{VG}
G.~Vranceanu and T.~Ganea,
``Topological embeddings of lens spaces'',
{\em Proc. Cambridge Philos. Soc.} {\bf 57} (1961) 688--690.

\bibitem{W}
M.~S.~Weiss,
``Embeddings from the point of view of immersion theory, Part I'',
{\em Geom. Topol.} {\bf 3} (1999) 67-101.

\bibitem{whitehead}
G.~W.~Whitehead,
``On mappings into group-like spaces'',
{\em Comment. Math. Helv.} {\bf 28} (1954) 320--328.

\bibitem{yasui1}
T.~Yasui,
``Note on the enumeration of embeddings of real projective spaces'',
{\em Hiroshima Math. J.} {\bf 3}  (1973) 409--418. 

\bibitem{yasui2}
T.~Yasui,
``Note on the enumeration of embeddings of real projective spaces II'',
{\em Hiroshima Math. J.} {\bf 6} (1976) 221--225. 

\bibitem{yo}
G.~Yo (Yueh Ching-chung),
``Cohomology mod $p$ of deleted cyclic product of a manifold'', 
{\em Sci. Sinica} {\bf 12} (1963) 1779-1794.

\end{thebibliography}
\end{document}